\documentclass[12pt,reqno]{amsart}
\usepackage{hyperref}
\usepackage{amscd,amssymb,amsfonts,amsbsy,amsmath,verbatim,color, mathrsfs}
\usepackage{pst-node}
\usepackage{tikz-cd}
\usepackage{graphicx,cite}

\usepackage{float}
\usepackage{tikz,ifthen,bm}
\usepackage{verbatim,mathdots,caption,subcaption,epstopdf}
\usetikzlibrary{intersections}
\usetikzlibrary{calc}

\newtheorem{thm}{Theorem}[section]

\newtheorem{lem}[thm]{Lemma}

\newtheorem{prop}[thm]{Proposition}
\newtheorem{exam}{Example}[section]
\theoremstyle{definition}
\newtheorem{defi}{Definition}[section]

\numberwithin{equation}{section}
\DeclareMathOperator{\ess}{ess}

\newcommand{\Real}{\mathbb{R}}

\newcommand{\R}{\Real}

\def\supp{{\textup{\textrm{supp}}}}

\makeatletter

\newcommand{\Rmnum}[1]{\expandafter\@slowromancap\romannumeral #1@}
\makeatother
\begin{document}
	\baselineskip=17pt
	\setcounter{figure}{0}
	
	\title[Hodge Theorem for Kre\u{\i}n-Feller operators]
	{Hodge Theorem for Kre\u{\i}n-Feller operators on compact Riemannian manifolds}
	\date{\today}
	\author[S.-M. Ngai]{Sze-Man Ngai}
	\address{Key Laboratory of High Performance Computing and Stochastic Information
		Processing (HPCSIP) (Ministry of Education of China), College of Mathematics and Statistics, Hunan Normal University,
		Changsha, Hunan 410081, China, and Department of Mathematical Sciences\\ Georgia Southern
		University\\ Statesboro, GA 30460-8093, USA.}
	\email{smngai@georgiasouthern.edu}
	
	\author[L. Ouyang]{Lei Ouyang}
	\address{Key Laboratory of High Performance Computing and Stochastic Information
		Processing (HPCSIP) (Ministry of Education of China, College of Mathematics and Statistics, Hunan Normal University, Changsha, Hunan 410081, China.}
	\email{ouyanglei@hunnu.edu.cn}


	\thanks{The authors were supported in part by the National Natural Science Foundation of China grant 12271156 and Construct Program of the Key Discipline in Hunan Province.}

	\begin{abstract}
		For an open set $\Omega$ in a compact smooth oriented Riemannian $n$-manifold and a positive finite Borel measure $\mu$ with support contained in $\overline{\Omega}$, we define an associated Kre\u{\i}n-Feller operator $\Delta_\mu^k$ on $k$-forms by assuming the Poincar\'e inequality. Kre\u{\i}n-Feller operators on $\R^n$ have been studied extensively in fractal geometry. Using results established by the authors, we obtain sufficient conditions for $\Delta_{\mu}^k$ to have compact resolvent. Under these conditions, we prove the Hodge theorem for forms, which states that there exists an orthonormal basis of $L^2\big(\bigwedge^k\Omega,\mu\big)$ consisting of eigenforms of $-\Delta_\mu^k$, the eigenspaces are finite-dimensional, and the eigenvalues of $-\Delta_\mu^k$  are real, countable, and increasing to infinity. One of these sufficient conditions is that the $L^\infty$-dimension  of $\mu$ is greater than $n-2$. Our result extends the classical Hodge theorem to Kre\u{\i}n-Feller operators.
	\end{abstract}
	
	\maketitle
	\section{Introduction}\label{S:IN}
	\setcounter{equation}{0}
	
	For a bounded open set $\Omega\subseteq\R^n$ and a positive finite Borel measure $\mu$ with support contained in $\overline{\Omega}$,  if $\mu$ satisfies the Poincar\'e inequality on $\Omega$, then there exists a Dirichlet Laplacian $\Delta_\mu$ defined by $\mu$ especially in connection with fractal geometry (see, e.g., \cite{Hu-Lau-Ngai_2006}). These Laplacians are also known as Kre\u{\i}n-Feller operators. The  properties of the Kre\u{\i}n-Feller operators $\Delta_\mu$ have been studied extensively by many authors, because of its connection with analysis on fractals
	(see \cite{Hu-Lau-Ngai_2006, Freiberg_2003, Freiberg_2005, Freiberg_2011, Freiberg-Zahle_2002, Fujita_1987, Ngai_2011, Ngai-Tang-Xie_2018, Ngai-Tang-Xie_2020, Ngai-Xie_2020, Ngai-Xie_2021, Bird-Ngai-Teplyaev_2003, Chen-Ngai_2010, Deng-Ngai_2015, Pinasco-Scarola_2019, Pinasco-Scarola_2021, Kessebohmer-Niemann_2022-2,Tang-Ngai_2022,Kessebohmer-Niemann_2022-3} and the references therein). For the classical Laplacian $\Delta$ defined on functions on Riemannian manifolds, eigenfunctions, eigenvalues, eigenvalue estimates, and heat kernel estimates have also been studied by many authors and have played important roles in geometric analysis (see, e.g., \cite{Li-Yau_1986,Minakshisundaram-Pleijel_1949,Schoen-Yau_1994,Grigor'yan_2009,Peter_2012,Chavel_1984} and references therein). The Laplacian $\Delta$ extends from functions (i.e., $0$-forms) to differential forms of arbitrary degree, namely, the Laplace-Beltrami operator $\Delta^k:=dd^*+d^*d:\Gamma^\infty(\bigwedge^k M)\rightarrow \Gamma^\infty(\bigwedge^k M)$, where the definitions of $d$, $d^*$, and $\Gamma^\infty(\bigwedge^k M)$ are given in Section \ref{S:Pre}. The spectrum of the Laplace-Beltrami operator $\Delta^k$ on differential forms has been investigated by many authors (see \cite{Chanillo-Treves_1997,Mantuano_2008,Gordon-Rossetti_2003} and references therein). In the early 1930s, Cartan discovered that the geometric and topological invariants of manifolds could be obtained directly from the calculation of differential forms, thus finding a deep connection between analysis and topology. In 1931, de Rham proved that de Rham cohomology groups are isomorphic to singular cohomology groups. This theorem gives the relationship between topology and smooth structures on manifolds. In 1933, Hodge established the Hodge theory while studying algebraic geometry. Hodge's theory took de Rham's work on de Rham's cohomology one step further. Hodge proved that there is a unique harmonic form in every de Rham cohomology class (see \cite{Hodge_1941}). Since harmonic forms are solutions to elliptic partial differential equations on manifolds, this theorem establishes fundamental connections between analysis, geometry, and topology on manifolds. In 1964, Atiyah and Bott defined elliptic complex as a generalization of de Rham complex. The fundamental theorem of elliptic complexes is also a generalization of the Hodge theorem (see, e.g., \cite[Theorem 5.2]{Wells_2008}). The validity  of the Hodge theorem for forms is related to the Hodge theorem and the fundamental theorem concerning elliptic complexes. The Hodge theorem for forms can be stated as follows.  Let $(M,g)$ be a compact connected oriented Riemannian manifold. There exists an orthonormal basis of $L^2(\bigwedge^kM)$ consisting of eigenforms of the Laplacian on $k$-forms $\Delta^k$. All the eigenvalues are nonnegative, each eigenvalue has finite multiplicity, and the eigenvalues accumulate only at infinity (see, e.g., \cite[Theorem 1.30]{Rosenberg_1997}). The authors extended Hodge theorem for functions to Laplacian $\Delta_\mu$ (see \cite[Theorem 2.2]{Ngai-Ouyang_2023}). Hodge's theorem plays an important role in the spectral theory on Riemannian manifolds and motivated this work.
		
		 In this paper, we study the Hodge theorem for Kre\u{\i}n-Feller operators defined on a compact oriented smooth Riemannian $n$-manifold. Our first goal is to
		  generalize the Kre\u{\i}n-Feller operator $\Delta_{\mu}$  on functions on a Riemannian manifold to an operator $\Delta^k_\mu$ on $k$-forms on a Riemannian manifold.	Our second goal is to prove an analog of Hodge's theorem for $\Delta_\mu^k$ acting on forms. Our third goal is to prove that if $M$ is compact and $\mu$ is absolutely continuous with respect to the Riemannian volume form and has a positive and bounded density, then each cohomology class of  forms in $W^{1,2}_0(\bigwedge^kM)$ contains a unique $\mu$-harmonic $k$-form.

We have the following main theorems.  The definitions of (PI),  $\underline{\operatorname{dim}}_{\infty}(\mu)$, $\operatorname{dom}(\mathcal{E})$, $\Delta_{\mu}^k$, and $\operatorname{dom}(\Delta_{\mu}^k)$ in the following three theorems are given in Section \ref{S:Pre}.
\begin{thm}\label{thm:1.1}
	Let $n\geq 1$, $M$ be a compact connected Riemannian $n$-manifold, and $\Omega \subseteq M$ be an open set. Let $\mu$ be a positive finite Borel measure on $M$ such that $\operatorname{supp}(\mu) \subseteq \overline{\Omega}$ and $\mu(\Omega)>0$. Assume that $\underline{\operatorname{dim}}_{\infty}(\mu)>n-2$. Then (PI) holds. Moreover, the embedding $\operatorname{dom}(\mathcal{E}) \hookrightarrow L^2\big(\bigwedge^k\Omega,\mu\big)$ is compact, where $0\leq k\leq n$.
\end{thm}
As a result, we have the following theorem.
\begin{thm}\label{thm:1.2}
	Assume the same hypotheses of Theorem \ref{thm:1.1}.  Then there exists an
	orthonormal basis of $L^2\big(\bigwedge^k\Omega,\mu\big)$ consisting of  eigenforms of $\Delta_{\mu}^k$, where $0\leq k\leq n$. The eigenvalues $\left\{\lambda_{m}\right\}_{m=1}^{\infty}$ satisfy $0\leq\lambda_{1} \leq \lambda_{2} \leq \cdots$, and if  $\dim(\operatorname{dom}(\mathcal{E})) =\infty$, then $\lim _{m \rightarrow \infty} \lambda_{m}=\infty$. Moreover, each eigenspace is finite-dimensional.
\end{thm}

\begin{defi}
Let $M$ be a Riemannian $n$-manifold and let $\mu$ be a positive finite Borel measure on $M$ such that $\operatorname{supp}(\mu) \subseteq M$. The {\em space of harmonic $k$-fields} is defined as
	\begin{align*}
		\widetilde{\mathcal{H}}^k(M):=\Big\{\omega\in W_0^{1,2}\big({\bigwedge}^kM\big):d\omega=d^*\omega=0\Big\},
	\end{align*}
	where $d$ and $d^*$ are defined in \eqref{eq:d} and \eqref{eq:ad}, respectively. The {\em space of harmonic $k$-forms}, denoted $\mathcal{H}^k(M)$,  is defined as
\begin{align*}
\mathcal{H}^k(M):=\Big\{\omega\in \operatorname{dom}(\Delta^k):\Delta^k\omega=0\Big\}.
\end{align*}
	Let
\begin{align*}
	\mathcal{H}^k_{\mu}(M):=\big\{\omega\in \operatorname{dom}(\Delta_{\mu}^k):\Delta_{\mu}^k\omega=0\big\}.
\end{align*}
The element of $\mathcal{H}^k_{\mu}(M)$ are called {\em $\mu$-harmonic $k$-forms}.

	\end{defi}
Assume that $\Omega=M$ and   $\mu$ is absolutely continuous with a positive density. The following theorem generalizes the classical Hodge theorem.

\begin{thm}\label{thm:1.3}
	Let $M$ be a compact Riemannian $n$-manifold. Let $\mu$ be a positive finite measure that is absolutely continuous with respect to the Riemannian volume form and has a positive and bounded  density. Then each cohomology class of  forms in $W^{1,2}_0(\bigwedge^kM)$ contains a unique $\mu$-harmonic $k$-form.
\end{thm}
In Example \ref{exam:1}, we show that the conclusion of Theorem \ref{thm:1.3} need not hold if $\mu$ is not absolutely continuous with respect to the Riemannian volume form.

The rest of this paper is organized as follows. In Section \ref{S:Pre}, we summarize some notation and definitions that will be needed throughout the paper. Moreover, we define the Kre\u{\i}n-Feller operator $\Delta_\mu^k$ on $k$-forms on a Riemannian manifold by assuming the Poincar\'e inequality and study its basic properties.
 In Section \ref{S:L}, we generalize the compact embedding theorem of Maz'ja\cite{Maz'ja_1985} to differential forms on a Riemannian manifold. Moreover, we use $\underline{\dim}_{\infty}(\mu)$ to study the compactness of the embedding $\operatorname{dom}(\mathcal{E})\hookrightarrow L^{2}(\bigwedge^k\Omega, \mu)$, and prove Theorems \ref{thm:1.1} and \ref{thm:1.2}. In Section \ref{S:H}, we prove Theorem \ref{thm:1.3}.

	\section{ Preliminaries}\label{S:Pre}
	\setcounter{equation}{0}
	In this paper, all Riemannian manifolds are assumed to be smooth and oriented. Also, whenever the Riemannian
	distance function is involved, we assume that the manifold is connected. Let $(M,g)$ be a Riemannian $n$-manifold with Riemannian metric $g$. 	Let $\mu$ be a positive bounded regular Borel measure on $M$ with ${\rm supp}(\mu)\subseteq\overline{\Omega}$. In this section, we summarize some notation and definitions that
	will be used throughout this paper. Moreover, by assuming that $\mu$ satisfies the Poincar\'e inequality, we will define the Laplacian $\Delta_{\mu}^k$.
	\subsection{Notation}\label{S:No}
	We summarize some notation needed in the paper, details can be founded in \cite{Jost,Tu_2011,Petersen_2016,Warner_1983}.  Let $(M,g)$ be a Riemannian $n$-manifold with Riemannian metric $g$. Let $dV_g$ be the Riemannian volume form. In any oriented smooth coordinate chart $(U,(x^i))$, the Riemannian volume form has the following local coordinate expression on $U$:
	$$\,dV_g=\sqrt{\det g_{ij}}\, dx^1\wedge\cdots\wedge dx^n,$$
	where $\wedge$ 
	denotes the exterior product and the $g_{ij}$ are the components of $g$ in these coordinates.  For any $p\in M$, we let $T_pM$ be the {\em tangent space} of $M$ at $p$ and let  $TM:=\bigcup_{p\in M}T_p M$. Denote the {\em cotangent space} of $M$ at $p$ by $T^*_pM$. Let $${\bigwedge}^k T^*_pM:=\underbrace{T^*_pM\wedge\cdots\wedge T^*_pM}_{k\,\,\text{times}}$$
		and let ${\bigwedge}^k T^*M:=\bigcup_{p\in M}{\bigwedge}^k T^*_pM$.
		If $\pi: E\rightarrow M$ is  a vector bundle (resp. $C^\infty$ vector bundle), then the {\em vector space of sections} (resp. {\em vector space of $C^\infty$ sections})
		of $E$ is denoted by $\Gamma(E)$ (resp. $\Gamma^\infty(E)$).  Let $\bigwedge^kM$ denote the {\rm $k$-th exterior power of the cotangent bundle}. The space of sections (resp. $C^\infty$ sections) of $\bigwedge^kM$ is denoted by $\Gamma(\bigwedge^kM)$ (resp. $\Gamma^\infty(\bigwedge^kM)$).  Elements in  $\Gamma(\bigwedge^kM)$ (resp. $\Gamma^\infty(\bigwedge^kM)$) are called {\em  $k$-forms} (resp. {\em smooth $k$-forms}) on $M$.  For simplicity, we let
	$$\mathcal{I}_{k,n} := \{I:=(i_1,\ldots,i_k)|1\leq i_1<i_2<\cdots<i_k\leq n\}$$
	be the set of all strictly ascending multi-indices between 1 and $n$ of length $k$. Let $(U,x^1,\ldots,x^n)$ be a coordinate chart on $M$. Then every $k$-form $\omega\in \Gamma(\bigwedge^kM)$ can be written as a linear combination  
	\begin{align}\label{eq:1.1}
		\omega=\sum_{I\in \mathcal{I}_{k,n}} \omega_I\,dx^I,
	\end{align}
	where the coefficients $\omega_I:U\rightarrow \R$ are the components of $\omega$ in the coordinate chart and  $\{dx^I\}$ is an orthonormal basis for $\bigwedge^kT^*U$. A $k$-form $\omega=\sum_{I\in \mathcal{I}_{k,n}} \omega_I\,dx^I$ on $U$ is smooth if and only if the coefficient
	functions $\omega_I$ are smooth on $U$; a $k$-form $\omega$ on $M$ is smooth if and only if its restriction to any chart $(U,x^1,\ldots,x^n)$ is  smooth (see, e.g., \cite[Proposition 18.8]{Tu_2011}). Let $(U,x^1,\ldots,x^n)$ be a coordinate chart on $M$. The {\em exterior derivative} $d:\Gamma^\infty(\bigwedge^kM)\rightarrow \Gamma^\infty(\bigwedge^{k+1}M)$ is defined locally on $U$ by the formula
\begin{align*}
	d\omega:=\sum_{r=1}^n\sum_{I\in \mathcal{I}_{k,n}}\frac{\partial \omega_I}{\partial x^r}dx^r\wedge dx^I.
\end{align*}
 If $k\geq n$, we let $d\omega=0$. Note that $(d\omega)_p$ is independent of the choice of $U$
 	containing $p$ (see, e.g., \cite[Section 19]{Tu_2011}).

Define the {\em Hodge star operator} $\star:\bigwedge^kT_p^*M\rightarrow \bigwedge^{n-k}T_p^*M$ as a pointwise isometry as
	$$\star(e^{i_1}\wedge\cdots\wedge e^{i_k})=e^{j_1}\wedge\cdots\wedge e^{j_{n-k}},\,\,0\leq k\leq n,$$
where $e^{i_1}\wedge\cdots\wedge e^{i_k}\wedge e^{j_1}\wedge\cdots\wedge e^{j_{n-k}}=e^{1}\wedge\cdots\wedge e^{n}$ and $\{e^{1}\wedge\cdots\wedge e^{n}\}$ is a positively oriented orthonormal basis of $T_p^*M$. 
For $\omega,\eta\in \Gamma^\infty(\bigwedge^kM)$, define the {\em pointwise inner product} of $\omega$ and $\eta$ as
$$\langle \omega, \eta\rangle_g:=\star(\omega\wedge\star\eta),\,\,\text{i.e.,}\,\,\langle \omega(p), \eta(p)\rangle_g:=\star(\omega(p)\wedge\star\eta(p))\,\,\text{for all}\,\,p\in M.$$
For $\omega\in \Gamma^\infty(\bigwedge^kM)$, let $|\omega|:=\langle \omega, \omega\rangle_g^{1/2}$ be {\em the pointwise  norm} of $\omega$. For $\omega,\eta\in \Gamma^\infty(\bigwedge^kM)$,  define the {\em inner product on $\Gamma^\infty(\bigwedge^kM)$} 
	\begin{align*}
		\left\langle \omega, \eta\right\rangle:=\int_M \left\langle \omega, \eta\right\rangle_g\,d V_g.
	\end{align*}
 The inner product structure on $\Gamma^\infty(\bigwedge^kM)$ allows us to define the {\em adjoint} $d^*:\Gamma^\infty(\bigwedge^{k+1}M)\rightarrow \Gamma^\infty(\bigwedge^kM)$
of $d$ via the formula
\begin{align*}
\langle d^*\omega, \eta\rangle=\langle \omega, d\eta\rangle
\end{align*}
for all $\eta\in \Gamma^\infty(\bigwedge^kM)$, where $0\leq k\leq n$.
Define the {\em support} of $\omega\in \Gamma^\infty(\bigwedge^kM)$ as the closure of the set $\{p\in M: \omega(p)\neq 0\}$.
	
	 Let $\Omega\subseteq M$ be an open set. Let $\overline{\Omega}$ and $\partial \Omega$ (possibly empty) be the closure and boundary of $\Omega$, respectively. Note that if $\Omega=M$, then $\partial \Omega=\emptyset$. Let $\Gamma^C(\bigwedge^k\Omega)$, $\Gamma^{\infty}(\bigwedge^k\Omega)$, and $\Gamma_c^{\infty}(\bigwedge^k\Omega)$ denote, respectively,  continuous $k$-forms,  $C^{\infty}$ $k$-forms, and $C^{\infty}$ $k$-forms with compact support.

Let $L^2(\bigwedge^k\Omega)$ denote the space of measurable $k$-forms on $\Omega$ satisfying $\int_{\Omega}|\omega|^2\,dV_g<\infty$, with the inner product $\langle\cdot,\cdot\rangle$ and associated norm $\|\cdot\|_{L^2(\bigwedge^k\Omega)}$ defined respectively as 
	\begin{align*}
		\left\langle \omega, \eta\right\rangle_{L^2(\bigwedge^k\Omega)}:=\int_\Omega \left\langle \omega, \eta\right\rangle_g\,dV_g\quad\text{and}\quad \|\omega\|_{L^2(\bigwedge^k\Omega)}:=\Big(\int_{\Omega}|\omega|^2\,dV_g\Big)^{1/2},
\end{align*}
 where $\omega$ is measurable with respect to the Riemannian volume form and $|\omega|$ is defined a.e. in $\Omega$. Let $\mu$ be a positive finite Borel measure on $M$ with  ${\rm supp}(\mu)\subseteq\overline{\Omega}$. Let $L^2\big(\bigwedge^k\Omega,\mu\big)$ denote the space of $k$-forms  measurable with respect to $\mu$ on $\Omega$ satisfying $\int_{\Omega}|\omega|^2\,d\mu<\infty$, with the inner product $\langle\cdot,\cdot\rangle_\mu$ and associated norm $\|\cdot\|_{L^2(\bigwedge^k\Omega,\mu)}$ defined respectively as
\begin{align*}
	\left\langle \omega, \eta\right\rangle_{\mu}:=\int_\Omega \left\langle \omega, \eta\right\rangle_g\,d\mu\quad\text{and}\quad \|\omega\|_{L^2(\bigwedge^k\Omega,\mu)}:=\Big(\int_{\Omega}|\omega|^2\,d\mu\Big)^{1/2}.
\end{align*}
 \begin{comment}Let $w_1=w_2$ $\mu$-a.e.. Then we can write $w_1=\sum_Iw_1^Idx^I=\sum_Iw_2^Idx^I=w_2$ $\mu$-a.e.. Hence $w_1^I=w_2^I$ $\mu$-a.e. for all $I$. Thus $\langle w_1, w_1\rangle_g=\sum_I|w_1^I|^2=\sum_I|w_2^I|^2=\langle w_2, w_2\rangle_g$ $\mu$-a.e.. Hence this inner product is well-defined.
\end{comment}
Then $L^2\big(\bigwedge^k\Omega,\mu\big)$ is a Hilbert space.

	We refer the reader to \cite{Morrey_1966,Scott_1995} for the following  definitions. We write $\Omega'\subset\subset \Omega$ if $\Omega'$ is an open subset of $\Omega$ and its closure $\overline{\Omega'}$ is contained on $\Omega$.  Define
	\begin{align*}
	L^1_{\rm{loc}}\Big({\bigwedge}^k\Omega\Big):=\Big\{u\in \Gamma\Big({\bigwedge}^k\Omega\Big):u\,\, \text{is measurable on}\,\,\Omega \,\, \text{and for any}&\\ \Omega'\subset \subset\Omega, u\in L^1\Big({\bigwedge}^k\Omega'\Big)&\Big\}.
\end{align*}
	Given a $k$-form $\omega=\sum_{I\in \mathcal{I}_{k,n}}\omega_Idx^I\in L^1_{\rm{loc}}(\bigwedge^kM)$,  we say that $\omega$ has a {\em generalized gradient} if, for each coordinate system, the pullbacks of the coordinate functions $\omega_I$ have generalized gradients in the sense that for all $\varphi\in C^\infty_c(\R^n)$, 
		\begin{align}\label{eq:g1}
		\int_{\R^n}\omega_I(x)\frac{\partial^\alpha\varphi}{\partial x^\alpha}(x)dx=(-1)^{|\alpha|}\int_{\R^n} \frac{\partial^\alpha \omega_I}{\partial x^\alpha}(x)\varphi(x)dx.
	\end{align}
 Let
\begin{align}\label{eq:so1}
	W\Big({\bigwedge}^k \Omega\Big):=\Big\{\omega\in L^1_{\rm{loc}}\Big({\bigwedge}^k\Omega\Big):\omega\,\,\text{has a generalized gradient}\Big\}.
	\end{align}
Now choose an
	atlas $\mathcal{A}$ for $\Omega$. For $(U, \varphi):=(U,(x^1,\ldots,x^n))\in \mathcal{A}$, write $\omega=\sum_{I\in \mathcal{I}_{k,n}}\omega_Idx^I\in L^1_{\rm{loc}}(\bigwedge^k\Omega)$. Define the
	{\em local gradient modulus} as
	\begin{align*}
		|\nabla_U\omega(x)|^2:=\sum_{I,l}\bigg|\frac{\partial \omega_I}{\partial x^l}(x)\bigg|^2
	\end{align*}
	and the {\em global gradient modulus} as
	\begin{align*}
		|\nabla\omega(x)|^2:=\sum_{U\in \mathcal{A}}\big|\nabla_U\omega(x)\big|^2.
	\end{align*}
	If  $\mathcal{A}$ is a locally finite
		cover of $\Omega$, then we may define the {\em (classical) Sobolev space} as
		\begin{align*}
			W^{1,2}_{\mathcal{A}}\Big({\bigwedge}^k\Omega\Big):=\Big\{\omega\in \Gamma\Big({\bigwedge}^k\Omega\Big):|\omega|,\,|\nabla\omega|\in L^2(\Omega)\Big\},
		\end{align*}
		with norm $\|\omega\|_{L^2(\bigwedge^k\Omega)}+\|\nabla\omega\|_{L^2(\bigwedge^k\Omega)}$.
		
		We now specify a class of atlases that yield equivalent Sobolev spaces. We say that 
		a coordinate system $(U, \varphi)$ is {\em regular}, if  $\overline{U}$ is compact and there is another
		system $(V, \psi)$ with $\overline{U}\subseteq V$ and $\psi|_U=\varphi$.

	Let $M$ be a compact Riemannian $n$-manifold. We may select finitely
	many systems $\{(V_i,\varphi_i)\}_{i=1}^N$ satisfying $M\subseteq \cup_{i=1}^NV_\alpha$.  By relabeling the $V_i$, we may choose a regular atlas $\{(U_i,\varphi_i|_{U_i})\}_{i=1}^N$ on $M$ such that $\overline{U}_{i}\subseteq V_{i}$ and $M\subseteq\cup_{i=1}^N U_i$.	\begin{comment} In fact, we first choose a family open sets $\{V_\alpha\}_{\alpha=1}^{N_1}$ satisfying $M\subseteq \cup_{\alpha=1}^{N_1}V_\alpha$. Then there exist countably many open ball $\{U_\beta\}_{\beta=1}^{M_\alpha}\subseteq  V_{\alpha}$ such that for each $\beta$,  $\overline{U}_{\beta}\subseteq V_{\alpha}$ and $M\subseteq\cup_{\alpha=1}^{N_1}\cup_{\beta=1}^{M_\alpha} U_\beta$. Using the compactness of $M$, we can select finitely many open ball $\{U_i\}_{i=1}^{N}\subseteq \{U_\beta\}_{\beta=1,\alpha=1}^{M_\alpha,N_1}$ such that for each $i$, $\overline{U}_{i}\subseteq V_{i}$ and $M\subseteq\cup_{i=1}^{N} U_i$, where $\{V_i\}_{i=1}^N= \{V_\alpha\}_{\alpha=1}^{N_1}$ and $V_i=V_j$ for some $i,j$. Hence there exists a regular atlas $\{(U_i,\varphi_i)\}_{i=1}^N$ on $M$. 
		\end{comment}

For $\omega=\sum_{I\in \mathcal{I}_{k,n}}\omega_Idx^I\in W({\bigwedge}^k\Omega)$, where $0\leq k<n$, we define its 
		{\em exterior derivative $d\omega$} as the $(k+1)$-form defined locally as
	\begin{align}\label{eq:d}
				d\omega:=\sum_{r=1}^n\sum_{I\in \mathcal{I}_{k,n}}\frac{\partial \omega_I}{\partial x^r}dx^r\wedge dx^I,
	\end{align}
where $\partial \omega_I/\partial x^r$ is defined as in \eqref{eq:g1}. For $k\geq n$, we let $d\omega=0$. %{\color{red}Note that $d^2\omega=0$ (see \cite[Proposition 6.8]{Scott_1995})}. 
For  $\omega=\sum_{I\in \mathcal{I}_{k,n}}\omega_Idx^I\in  W({\bigwedge}^k\Omega)$ and $k\geq 1$, we 
define its adjoint $d^*\omega$ by the condition that 
\begin{align}\label{eq:ad}
	\langle d^*\omega, \eta\rangle_{L^2(\bigwedge^k\Omega)}=\langle \omega, d\eta\rangle_{L^2(\bigwedge^k\Omega)}\qquad \text{for all}\,\,\eta\in W\Big({\bigwedge}^k\Omega\Big).
\end{align}
 If $\omega$ is a $0$-form, we define $d^*\omega =0$.
The {\em Sobolev space of $k$-forms} on $M$ is defined as
	\begin{align*}
		W^{1,2}\Big({\bigwedge}^k\Omega\Big):=\Big\{\omega\in W\Big({\bigwedge}^k \Omega\Big)\bigcap L^2\Big({\bigwedge}^k \Omega\Big):d\omega\in L^2\Big({\bigwedge}^{k+1} \Omega\Big)\\\text{and}\,\,d^*\omega\in L^2\Big({\bigwedge}^{k-1} \Omega\Big)\Big\}
	\end{align*}
	equipped with the inner product $\langle\cdot,\cdot\rangle_{W^{1,2}(\bigwedge^k\Omega)}$ associated to $\|\cdot\|_{W^{1,2}(\bigwedge^k\Omega)}$ defined as
	$$ \left\langle \omega, \eta\right\rangle_{W^{1,2}(\bigwedge^k\Omega)} :=\left\langle \omega, \eta\right\rangle_{L^2(\bigwedge^k\Omega)}+\left\langle d\omega, d\eta\right\rangle_{L^2(\bigwedge^k\Omega)}+\left\langle d^*\omega, d^*\eta\right\rangle_{L^2(\bigwedge^k\Omega)}.$$
	The norm associated to $\langle\cdot,\cdot\rangle_{W^{1,2}(\bigwedge^k\Omega)}$ is defined as
	\begin{align*}
		\|u\|_{W^{1,2}(\bigwedge^k\Omega)}:=\Big(\int_{\Omega}|u|^2\,dV_g+\int_{\Omega}|d u|^2\,dV_g+\int_{\Omega}|d^* u|^2\,dV_g\Big)^{\frac{1}{2}}.
	\end{align*}
	Let $W^{1,2}_0(\bigwedge^k\Omega)$ denote the closure of $\Gamma^\infty_c(\bigwedge^k\Omega)$ in the $W^{1,2}(\bigwedge^k\Omega)$ norm.

	It is known that regular atlases yield classical Sobolev spaces $W^{1,2}_{\mathcal{A}}(\bigwedge^k\Omega)$  equivalent to $W^{1,2}(\bigwedge^k\Omega)$ (see, e.g, \cite{Scott_1995, Morrey_1966}). %{\color{magenta}Hence \eqref{eq:a1} holds for $\omega\in W^{1,2}(\bigwedge^k\Omega)$ and $\eta\in W^{1,2}(\bigwedge^k\Omega)$.}
	
	We denote the {\em Euclidean distance} by $d_{E}(\cdot, \cdot)$. For a connected Riemannian $n$-manifold $M$, we denote the {\em Riemannian distance} by $d_M(\cdot,\cdot)$.  For any $\xi\in TM$, define the {\em length of $\xi$} as $|\xi|_E:=\langle \xi,\xi\rangle^{1/2}$.  For $\epsilon>0$, let
	\begin{align*}
		&B(x,\epsilon):=\{y\in \mathbb{R}^n: d_E(x,y)<\epsilon\}, \quad x\in \mathbb{R}^n,\\
		&B^M(p,\epsilon):=\{q\in M: d_M(p,q)<\epsilon\},\quad p\in M,\\
		&B^{T_pM}(0,\epsilon):=\{\xi\in T_pM: |\xi|_E<\epsilon\}.
	\end{align*}
	Let $\mu$ be a positive bounded regular Borel measure on $M$.  To state our main results, recall that the {\em lower} and {\em upper $L^{\infty}$-dimensions of $\mu$}
	are defined respectively as
	\begin{align*}
		\underline{\dim}_\infty(\mu):=\displaystyle{\varliminf_{\delta\to 0^+}}\frac{\ln (\sup_x \mu(B^{M}(x,\delta)))}{\ln \delta}\,\,\,\,\text{and}\,\,\,\,
		\overline{\dim}_\infty(\mu):=\displaystyle{\varlimsup_{\delta\to 0^+}}\frac{\ln (\sup_x \mu(B^{M}(x,\delta)))}{\ln \delta},
	\end{align*}
	where $B^{M}(x,\delta)$ is a $\delta$-ball with center $x\in{\rm supp}(\mu)$ and the supremum is taken over all $x\in{\rm supp}(\mu)$. If the limit exists, we denote the common value by $\dim_\infty(\mu)$. 

		\subsection{Kre\u{\i}n-Feller operators} \label{S:frac}
		Let $M$ be a compact Riemannian $n$-manifold and let $\Omega \subseteq M$ be an open set. Let $\mu$ be a positive bounded regular Borel measure on $M$ with ${\rm supp}(\mu)\subseteq\overline{\Omega}$.   Throughout this paper, we assume that $\mu(\Omega)>0$. Our method of defining a   Laplacian $\Delta_{\mu}^k$ on $L^2\big(\bigwedge^k\Omega,\mu\big)$ associated with $\mu$ is similar to that in \cite{Hu-Lau-Ngai_2006}. First, we need to assume that $\mu$ satisfies the following {\em Poincar$\acute{e}$ inequality:}
	\begin{enumerate}
		\item[(PI)] There exists a constant $C>0$ such that for all $u\in \Gamma^{\infty}_c(\bigwedge^k\Omega)$,
		\begin{align}\label{eq:PI}
			\int_\Omega |u|^2\,d\mu\leq C\int_\Omega \big(|d u|^2+ |d^*u|^2\big)\,dV_g.
		\end{align}
	\end{enumerate}
\begin{comment}For $k=0$, we have $|d u|^2+ |d^*u|^2=|\nabla u|^2$.
	\end{comment}
	This condition implies that each equivalence class $u \in W^{1,2}_0(\bigwedge^k\Omega)$ contains a unique (in the $L^2\big(\bigwedge^k\Omega,\mu\big)$ sense) member $\overline{u}$ that belongs to $L^{2}\big(\bigwedge^k\Omega, \mu\big)$ and satisfies both conditions below:
	\begin{enumerate}
		\item[(1)] There exists a sequence $\left\{u_m\right\}$ in $\Gamma_{c}^{\infty}(\bigwedge^k\Omega)$ such that $u_m \rightarrow \overline{u}$ in $W^{1,2}_0(\bigwedge^k\Omega)$ and $u_m \rightarrow \overline{u}$ in $L^{2}(\bigwedge^k\Omega, \mu)$;
		\item[(2)] $\overline{u}$ satisfies the inequality in \eqref{eq:PI}.
	\end{enumerate}

	We call $\overline{u}$ the {\em $L^2\big(\bigwedge^k\Omega,\mu\big)$-representative} of $u$. Assume  $\mu$ satisfies (PI) and define a mapping $\iota: W^{1,2}_0(\bigwedge^k\Omega) \rightarrow L^2\big(\bigwedge^k\Omega,\mu\big)$ by
	$$
	\iota(u)=\overline{u}.
	$$
	Next, we notice that $\iota$ is a bounded linear operator but is not necessarily injective. Hence we consider the subspace $\mathcal{N}_k$ of $W^{1,2}_0(\bigwedge^k\Omega)$ defined as
	$$
	\mathcal{N}_k:=\left\{u \in W^{1,2}_0\Big({\bigwedge}^k\Omega\Big):\|\iota(u)\|_{L^2(\bigwedge^k\Omega,\mu)}=0\right\}.$$
	Since $\mu$ satisfies (PI), $\mathcal{N}_k$ is a closed subspace of $W^{1,2}_0(\bigwedge^k\Omega)$. Let $\mathcal{N}_k^{\perp}$ be the orthogonal complement of $\mathcal{N}_k$ in $W^{1,2}_0(\bigwedge^k\Omega)$. Then $\iota: \mathcal{N}_k^{\perp} \rightarrow L^2\big(\bigwedge^k\Omega,\mu\big)$ is injective. Throughout this paper, if no confusion is possible, we will denote $\overline{u}$ simply by $u$. For all $u \in \mathcal{N}_k^{\perp}$, we see that $\|u\|_{L^2(\bigwedge^k\Omega,\mu)} \leq C^{1 / 2}\|u\|_{W^{1,2}_0(\bigwedge^k\Omega)}$ by (PI), i.e., $\mathcal{N}_k^{\perp}$ is embedded in $L^2\big(\bigwedge^k\Omega,\mu\big)$.
	%If $d \geqslant 2$ and if $\mu$ has a point mass in $\Omega$, then condition (C1) fails, since $W^{1,2}_0(\bigwedge^k\Omega)$ contains unbounded functions. We will study condition (C1) in detail in Section 3 for general measures and in Section 5 for self-similar measures.
	
	%We remark that condition (C1) is similar to a condition in [34, Chapter 1], which is defined under a different setting, e.g., $\operatorname{supp}(\mu)$ there is assumed to have zero Lebesgue measure and is contained in a $\Gamma^\infty$ domain $\Omega$.
	Now, we consider a nonnegative bilinear form $\mathcal{E}(\cdot, \cdot)$ on $L^2\big(\bigwedge^k\Omega,\mu\big)$ defined as
	\begin{align}\label{eq(1.1)}
		\mathcal{E}(u, v):=\int_\Omega \langle d u,d u\rangle_g+ \langle d^* u,d^* u\rangle_g\,dV_g
	\end{align}
	with \textit{domain} $\operatorname{dom}(\mathcal{E})$ equal to $\mathcal{N}_k^{\perp}$, or more precisely, $\iota(\mathcal{N}_k^{\perp})$. 

 We define another nonnegative bilinear form $\mathcal{E}_{*}(\cdot, \cdot)$ on $L^2\big(\bigwedge^k\Omega,\mu\big)$ as
	\begin{align}\label{eq(2.1)}
		\mathcal{E}_{*}(u, v):=\mathcal{E}(u, v)+\langle u, v\rangle_{\mu}=\int_{\Omega} \big(\langle d u, d v\rangle_g+ \langle d^* u, d^* v\rangle_g \big)\,dV_g+\int_{\Omega} \langle u, v\rangle_g \,d \mu.
	\end{align}
	Then $\mathcal{E}_{*}(\cdot, \cdot)$ is an inner product on $\operatorname{dom}(\mathcal{E})$. 
	
	\begin{comment}
		\begin{lem}\label{lem:2.1}
			\cite{Karcher_1987} A manifold is second countable if and only if it has a countable atlas.
		\end{lem}
		{\color{brown}
			\begin{proof}\,Let $V=\{V_n\}_{n\in \mathbb {N}}$ be a countable base. Take any atlas $(\kappa_{\alpha}, U_{\alpha})$ of $M$. Let $K=\{n\in \mathbb {N}| V_n$ is contained in some $U_{\alpha}\}$. This is a countable set. For $k\in K$ choose $\alpha_k$ such that $V_k\subset U_{\alpha k}$. Let $x\in M$. There exists $\alpha \in A$ such that $x\in U_{\alpha}$. Since $V$ is a base, there exists $k\in \mathbb {N}$ such that $x\in V_k\subset U_{\alpha k}$. But then $x\in U_{\alpha_k}$. Therefore the $(U_{\alpha_k})_{k\in K}$ cover M which gives you a countable atlas.
		\end{proof}}
	\end{comment}
	
	To prove Proposition \ref{prop:2.3}, we need the following lemma. 
	
	\begin{lem}\label{lem:2.2} Let $M$ be a compact Riemannian $n$-manifold and let $\Omega\subseteq M$ be an open set. Let $\mu$ be a positive finite Borel measure on $M$ such that ${\rm supp}(\mu) \subseteq \overline{\Omega}$ and $\mu(\Omega)>0$. Then $\Gamma^\infty_c(\bigwedge^k\Omega)$ is dense in $L^2\big(\bigwedge^k\Omega,\mu\big)$.
	\end{lem}
		\begin{proof}\,Let $(U,\varphi)$ be a coordinate chart on $M$ and $\omega\in L^2\big(\bigwedge^kU,\mu\big)$.  By \eqref{eq:1.1}, we write
			\begin{align*}
			\omega=\sum_{I\in \mathcal{I}_{k,n}} \omega_I\,dx^I,
		\end{align*}
		where the coefficients $\omega_I$ are measurable functions with respect to $\mu$ on $U$.
		Hence $$\int_{U}\sum_{I\in \mathcal{I}_{k,n}} |\omega_I|^2\,d\mu<\infty.$$ Thus for each $I$, $\omega_I\in L^2(U,\mu)$. Let $\epsilon>0$ be arbitrary. For each $I$, let $\epsilon_I>0$ such that 
			\begin{align}\label{eq:4.2}
				\sum_{I\in \mathcal{I}_{k,n}}\epsilon_I<\epsilon.
					\end{align}
			 Since
			$C^\infty_c(U)$ is dense in $L^2(U,\mu)$, for each $I$, there exists some $f_I\in C^\infty_c(U)$ such that 
			\begin{align}\label{eq:2.2}
				\|f_I-\omega_I\|_{L^2(U,\mu)} <\epsilon_I.
			\end{align}
			Let $f:=\sum_{I\in \mathcal{I}_{k,n}}f_Idx^I$. Then $f$ is smooth (see, e.g., \cite[Lemma 18.6]{Tu_2011}), and by definition, $f$ has compact support. Thus $f\in \Gamma^\infty_c(\bigwedge^kU)$. Moreover, 
			\begin{align*}
				\|f-\omega\|_{L^2(\bigwedge^kU,\mu)}^2&=\int_U\Big|\sum_{I\in \mathcal{I}_{k,n}}(f_I-\omega_I)dx^I\Big|^2\,d\mu\nonumber\\
				&<\sum_{I\in \mathcal{I}_{k,n}}\epsilon_I \qquad(\text{by} \,\,\eqref{eq:2.2})\nonumber\\
				&<\epsilon.\,\,\qquad\qquad(\text{by} \,\,\eqref{eq:4.2})
			\end{align*}
		\begin{comment}\begin{align*}
			\|f-\omega\|_{L^2(\bigwedge^kU,\mu)}^2&=\int_U|f-\omega|^2\,d\mu\nonumber\\
			&=\int_U\Big|\sum_{I\in \mathcal{I}_{k,n}}(f_I-\omega_I)dx^I\Big|^2\,d\mu\nonumber\\
		&=\int_U\sum_{I\in \mathcal{I}_{k,n}}|f_I-\omega_I|^2\,d\mu\nonumber\\	
			&\leq\sum_{I\in \mathcal{I}_{k,n}}\|f_I-\omega_I\|^2_{L^2(U,\mu)}\nonumber\\
			&<\sum_{I\in \mathcal{I}_{k,n}}\epsilon_I \qquad(\text{by} \,\,\eqref{eq:2.2})\nonumber\\
			&<\epsilon.\qquad(\text{by} \,\,\eqref{eq:4.2})
		\end{align*}
	\end{comment}
		Thus $\Gamma_{c}^{\infty}(\bigwedge^kU)$ is dense in $L^2(\bigwedge^kU,\mu)$.
		
		By the compactness of $M$, we can choose a $C^\infty$ atlas  $\{(U_i,\varphi_{U_i})\}_{i=1}^{N}$ on $M$ and a $C^\infty$ partition of unity $\{\rho_{i}\}_{i=1}^{N}$ satisfying:
		$$M\subseteq\bigcup_{i=1}^NU_i,\quad{\rm supp}(\rho_{i})\subseteq U_{i},\quad\text{and}\,\sum^{N}_{i=1}\rho_{i}(x)=1,\quad x\in M.$$
		Let $\omega\in L^2\big(\bigwedge^k\Omega,\mu\big)$. We can write $\omega:=\sum^{N}_{i=1}\rho_{i}\omega$, with ${\rm supp}(\rho_{i}\omega)\subseteq U_{i}$. Let $\epsilon>0$ be arbitrary. For each $i=1,\ldots, N$, there exists some $g_{\alpha_i}\in \Gamma_{c}^{\infty}(\bigwedge^kU_{i})$ such that
		\begin{align}\label{eq:pa}
		 \|\rho_{i}\omega-g_{\alpha_i}\|_{L^2(\bigwedge^kU_i,\mu)}<\frac{\epsilon}{2^i}.
		 \end{align}
		Let $g:=\sum^{N}_{i=1}g_{\alpha_i}$. Then $g\in \Gamma_{c}^{\infty}(\bigwedge^k\Omega)$. 
		Hence
		\begin{align*}
			\|\omega-g\|_{L^2(\bigwedge^k\Omega,\mu)}&=\bigg\|\sum^{N}_{i=1}\rho_{i}\omega-\sum^{N}_{i=1}g_{\alpha_i}\bigg\|_{L^2(\bigwedge^k\Omega,\mu)}\\
			&\leq \sum^{N}_{i=1}\|\rho_{i}\omega-g_{\alpha_i}\|_{L^2(\bigwedge^k\Omega,\mu)}\\
			&\leq \sum^{N}_{i=1}\frac{\epsilon} {2^i}<\epsilon.\qquad(\text{by}\,\,\eqref{eq:pa})
		\end{align*}
		Hence $\Gamma^\infty_c(\bigwedge^k\Omega)$ is dense in $L^2\big(\bigwedge^k\Omega,\mu\big)$.
\end{proof}

	\begin{prop}\label{prop:2.3}
	Let $M$ be a compact Riemannian $n$-manifold and let $\Omega\subseteq M$ be an open set. Let $\mu$ be a positive finite Borel measure on $M$ such that ${\rm supp}(\mu) \subseteq \overline{\Omega}$ and $\mu(\Omega)>0$. Let $\mathcal{E}$ and $\mathcal{E}_{*}$ be the quadratic forms defined as in (\ref{eq(1.1)}) and (\ref{eq(2.1)}), respectively. Assume  $\mu$ satisfies (PI). Then
		\begin{enumerate}
			\item[(a)] $\operatorname{dom}(\mathcal{E})$ is dense in $L^2\big(\bigwedge^k\Omega,\mu\big)$.
			\item[(b)] $\left(\mathcal{E}_{*}, \operatorname{dom}(\mathcal{E})\right)$ is a Hilbert space.
		\end{enumerate}
	\end{prop}
	\begin{proof}\,(a) By Lemma \ref{lem:2.2}, $\Gamma^\infty_c(\bigwedge^k\Omega)$ is dense in  $L^{2}(\bigwedge^k\Omega, \mu)$. Now let $u \in L^{2}(\bigwedge^k\Omega, \mu)$ and let $\left\{u_m\right\}$ be a sequence in $\Gamma^\infty_c(\bigwedge^k\Omega)$ converging to $u$ in the $L^{2}(\bigwedge^k\Omega, \mu)$-norm. Write $u_m=u_m^{0}+u_m^{\perp}$, where $u_m^{0} \in \mathcal{N}_k$ and $u_m^{\perp} \in \mathcal{N}_k^{\perp}$. Then  $\iota(u_m^{\perp})\in\operatorname{dom}(\mathcal{E})$ and $	\lim_{m \rightarrow \infty}\|\iota(u_m^{\perp})-u\|_{L^2(\bigwedge^k\Omega,\mu)}=0$.
	\begin{comment}\begin{align*}
		\|\iota(u_m^{\perp})-u\|_{L^2\big(\bigwedge^k\Omega,\mu\big)}&=\|\iota(u_m^{\perp})-\iota(u_m)+\iota(u_m)-u\|_{L^2\big(\bigwedge^k\Omega,\mu\big)}\\
		&\leq \|\iota(u_m^0)\|_{L^2\big(\bigwedge^k\Omega,\mu\big)}+\|u_m-u\|_{L^2\big(\bigwedge^k\Omega,\mu\big)}\rightarrow 0.	
	\end{align*}
\end{comment}
	Thus $\iota(u_m^{\perp} )$ converges to $u$ in $L^{2}(\bigwedge^k\Omega, \mu)$.
	Hence, $\operatorname{dom}(\mathcal{E})$ is dense in $L^{2}(\bigwedge^k\Omega, \mu)$.
	
	(b) Under (PI), we can show that the norm induced by $\mathcal{E}_{*}$ is equivalent to the norm $\|\cdot\|_{W^{1,2}_0(\bigwedge^k\Omega)}$. 	Hence $\left(\mathcal{E}_{*}, \operatorname{dom}(\mathcal{E})\right)$ is complete.
\end{proof}

	Proposition \ref{prop:2.3} implies that if $\mu$ satisfies (PI), then the quadratic form $(\mathcal{E}, \operatorname{dom}(\mathcal{E}))$ is closed on $L^2\big(\bigwedge^k\Omega,\mu\big)$.  Hence it follows from standard theory that there exists a nonnegative self-adjoint operator $-\Delta^k_\mu$ on $L^2\big(\bigwedge^k\Omega,\mu\big)$ with $\operatorname{dom}(\Delta^k_\mu) \subseteq \operatorname{dom}\Big({(\Delta^k_\mu)}^{1 / 2}\Big)=\operatorname{dom}(\mathcal{E})$, such that
	$$
	\mathcal{E}(u, v)=\left\langle-{(\Delta^k_\mu)}^{1 / 2} u, -{(\Delta^k_\mu)}^{1 / 2} v\right\rangle_{\mu} \quad \text { for all } u, v \in \operatorname{dom}(\mathcal{E}).
	$$
	Moreover, $u \in \operatorname{dom}(\Delta^k_\mu)$ if and only if $u \in \operatorname{dom}(\mathcal{E})$ and there exists $f \in L^2\big(\bigwedge^k\Omega,\mu\big)$ such that 
	\begin{align}\label{no1}
		\mathcal{E}(u, v)=\langle f, v\rangle_{\mu} \quad \text{for all}\,\, v \in \operatorname{dom}(\mathcal{E})
		\end{align}
		(see, e.g., \cite{Kigami_2001}). Note that for all $u \in \operatorname{dom}(\Delta^k_\mu)$ and $v \in \operatorname{dom}(\mathcal{E})$,
	\begin{align}\label{eq(2.2)}
		\int_{\Omega} \big(\langle du, dv\rangle_g+\langle d^*u, d^*v\rangle_g \big)\, dV_g=\mathcal{E}(u, v)=\langle -\Delta^k_\mu u, v\rangle_{\mu}.
	\end{align}
If $\mu=V_g$, then for all $u \in \operatorname{dom}(\Delta^k)$ and $v \in \operatorname{dom}(\mathcal{E})$,
		\begin{align*}
			\int_{\Omega} \big(\langle du, dv\rangle_g+\langle d^*u, d^*v\rangle_g \big)\, dV_g=\mathcal{E}(u, v)=\langle -\Delta^k u, v\rangle_{L^2(\bigwedge^k\Omega)}.
		\end{align*}

\begin{comment}
{\color{red}  Let $\mathcal{D}(\Omega)$ denote the space of {\em test $k$-forms} consisting of $\Gamma^\infty_c(\bigwedge^k\Omega)$ equipped with the following topology: a sequence $\left\{u_{m}\right\}$ converges to a $k$-form $u$ in $\mathcal{D}(\Omega)$ if there exists a compact $K \subseteq \Omega$ such that $\operatorname{supp}\left(u_{m}\right) \subseteq K$ for all $m$, and for any exterior derivative $d^{s}$ of order $s$, the sequence $\left\{d^{s} u_{m}\right\}$ converges to $d^{s} u$ uniformly on $K$. Denote by $\mathcal{D}'(\Omega)$ the space of distributions, the dual space of $\mathcal{D}(\Omega)$.}
\end{comment}

	\begin{prop}\label{prop:2.4}
		Assume that conditions (PI) holds. For $u \in \operatorname{dom}(\mathcal{E})$ and $f \in L^{2}(\bigwedge^k\Omega, \mu)$, the following conditions are equivalent:
		\begin{enumerate}
			\item[(a)] $u \in \operatorname{dom}(\Delta^k_\mu)$ and $-\Delta^k_\mu u=f$;
			\item[(b)] $-\Delta^k u=f d \mu$ in the sense that  for any $v \in \Gamma^\infty_c(\bigwedge^k\Omega)$,
			\begin{align}\label{eq(2.3)}
			\int_{\Omega}\big( \langle du, dv\rangle_g+\langle d^*u, d^*v\rangle_g\big) \, dV_g=\int_{\Omega} \langle v, f\rangle_g\, d \mu.
			\end{align}	
		\end{enumerate}
	\end{prop}
%\end{comment}
 %\begin{comment}
		\begin{proof}\, Assume that (a) holds. We see from (\ref{eq(2.2)}) that
			$$
			\int_{\Omega}\langle v, f\rangle_g d \mu=\langle -\Delta^k_\mu u, v\rangle_{\mu}=\mathcal{E}(u, v)=\int_{\Omega} \big( \langle du, dv\rangle_g+\langle d^*u, d^*v\rangle_g\big) \,dV_g
			$$
			for any $v \in  \Gamma^\infty_c(\bigwedge^k\Omega)$. Hence (b) holds.
			
			Conversely, assume that (b) holds. 
			Since $ \Gamma^\infty_c(\bigwedge^k\Omega)$ is dense in $\operatorname{dom}(\mathcal{E})$, one can show, by using condition (PI), that (\ref{eq(2.3)}) also holds for all $v \in \operatorname{dom}(\mathcal{E})$. In fact, for each $v\in  \Gamma^\infty_c(\bigwedge^k\Omega)$, there exists $v_m\in \operatorname{dom}(\mathcal{E})$ such that 
		\begin{align}\label{eq:dom}
			\lim_{m \rightarrow \infty}\big(\langle (v_m-v), (v_m-v)\rangle_{\mu}+ \langle d(v_m-v), d(v_m-v)\rangle_{L^2(\bigwedge^k\Omega)}&\nonumber\\+ \langle d^* (v_m-v), d^* (v_m-v)\rangle_{L^2(\bigwedge^k\Omega)}\big) &= 0.
		\end{align}
Using  \eqref{eq:dom}, we have
		\begin{align*}
		& 	\lim_{m \rightarrow \infty}\big|\langle du, dv_m\rangle_{L^2(\bigwedge^k\Omega)}+ \langle d^*u, d^*v_m\rangle_{L^2(\bigwedge^k\Omega)}- \langle du, dv\rangle_{L^2(\bigwedge^k\Omega)}- \langle d^*u, d^*v\rangle_{L^2(\bigwedge^k\Omega)}\big|\\
		\leq&	\lim_{m \rightarrow \infty}\Big|\langle du, du\rangle^{1/2}_{L^2(\bigwedge^k\Omega)}\langle d(v_m-v), d(v_m-v)\rangle^{1/2}_{L^2(\bigwedge^k\Omega)} \\&+\langle d^*u, d^*u\rangle^{1/2}_{L^2(\bigwedge^k\Omega)}\langle d^*(v_m-v), d^*(v_m-v)\rangle^{1/2}_{L^2(\bigwedge^k\Omega)}\Big|\\
		=&\,\, 0 
	\end{align*}
	and
\begin{align*}
		\lim_{m \rightarrow \infty}\big| \langle v, f\rangle_{\mu}- \langle v_m, f\rangle_{\mu}\big|
	\leq\lim_{m \rightarrow \infty}\big|\langle v-v_m, v-v_m\rangle_{\mu}^{1/2}\langle f, f\rangle_{\mu}^{1/2} \big|
	= 0. 
\end{align*}
\begin{comment}	\begin{align*}
				& \langle du, dv_m\rangle+ \langle d^*u, d^*v_m\rangle- \langle du, dv\rangle- \langle d^*u, d^*v\rangle\\
				=& \langle du, dv_m-dv\rangle+ \langle d^*u, d^*v_m-d^*v\rangle\\
				=& \langle du, d(v_m-v)\rangle+ \langle d^*u, d^*(v_m-v)\rangle\\
				\leq&\langle du, du\rangle^{1/2}\langle d(v_m-v), d(v_m-v)\rangle^{1/2} +\langle d^*u, d^*u\rangle^{1/2}\langle d^*(v_m-v), d^*(v_m-v)\rangle^{1/2}\\
				\rightarrow& 0\,(m\rightarrow \infty), \quad\quad(\text{by \eqref{eq:dom}})
			\end{align*}
			\begin{align*}
			& \langle v, f\rangle_{\mu}- \langle v_m, f\rangle_{\mu}\\
			=& \langle v-v_m, f\rangle_{\mu}\\
			\leq&\langle v-v_m, v-v_m\rangle_{\mu}^{1/2}\langle f, f\rangle_{\mu}^{1/2} \\
			\rightarrow& 0\,(m\rightarrow \infty). \quad\quad(\text{by \eqref{eq:dom}})
		\end{align*}
	\end{comment}
		Hence, we see that $\mathcal{E}(u, v)=$ $\langle f, v\rangle_{\mu}$ for all $v \in \operatorname{dom}(\mathcal{E})$. This implies that $u \in \operatorname{dom}(\Delta^k_\mu)$ and $-\Delta^k_\mu u=f$. Therefore, (a) follows.
	\end{proof}
	
We call $\Delta_{\mu}^k$ the {\em Laplacian or Kre\u{\i}n-Feller operator on $k$-forms with respect to} $\mu$. 
 For any $u \in \operatorname{dom}\left(\Delta_{\mu}^k\right)$, we have $\Delta^k u=\Delta_{\mu}^k u\, d \mu$ in the sense of distribution by Proposition \ref{prop:2.4}.

	\begin{thm}\label{thm:2.5}
		Assume that condition (PI) holds. Then for any $f \in L^{2}(\bigwedge^k\Omega, \mu)$, where $0\leq k\leq n$, there exists a
		unique $u \in \operatorname{dom}\left(\Delta_{\mu}^k\right)$ such that $\Delta_{\mu}^k u=f.$ The operator
		$$(\Delta_{\mu}^k)^{-1}: L^{2}\Big({\bigwedge}^k\Omega, \mu\Big) \rightarrow \operatorname{dom}(\Delta_{\mu}^k),\quad f\mapsto u $$
		is bounded and has norm at most $C$, the constant in (PI).
	\end{thm}
 The proof of Theorem \ref{thm:2.5} is similar to that of \cite[Theorem 2.3]{Hu-Lau-Ngai_2006}; we omit the details.
		\begin{comment}{\color{brown}\begin{proof}\,
			Let $f \in L^{2}(\bigwedge^k\Omega, \mu)$. Define a linear functional $T_f$ on $\operatorname{dom}(\mathcal{E})$ by
			$$T_f(v)=-\int_{\Omega}\langle v, f\rangle_g \,d\mu,\quad v\in \operatorname{dom}(\mathcal{E}).$$
			By (PI),
			$$
			\left|T_{f}(v)\right| \leq\|f\|_{L^{2}(\bigwedge^k\Omega, \mu)}\|v\|_{L^{2}(\bigwedge^k\Omega, \mu)} \leq C\|f\|_{L^{2}(\bigwedge^k\Omega, \mu)} \mathcal{E}(v, v)^{1 / 2}.
			$$
			Hence $T_{f}$ is continuous. By the Riesz representation theorem, there exists a unique $u \in \operatorname{dom}(\mathcal{E})$ such that
			\begin{align}\label{eq(2.4)}
				\|u\|_{W_{0}^{1,2}(\bigwedge^k\Omega)}=\left\|T_{f}\right\| \leq C\|f\|_{L^{2}(\bigwedge^k\Omega, \mu)}
			\end{align}
			and for all $v \in \operatorname{dom}(\mathcal{E})$,
			$$
			-\int_{\Omega}\langle v, f\rangle_g \,d\mu=T_{f}(v)=\mathcal{E}(u, v)
			$$
			Therefore $\Delta^k u=f d \mu$ in the sense of distribution. Proposition \ref{prop:2.4} implies that $u \in$ $\operatorname{dom}\left(\Delta_{\mu}^k\right)$ and $\Delta_{\mu}^k u=f$. The last assertion follows from (\ref{eq(2.4)}).
	\end{proof}}
\end{comment}
	
	Theorem \ref{thm:2.5} shows that for any $f \in L^{2}(\bigwedge^k\Omega, \mu)$, the equation
	$$
	\Delta_{\mu}^k u=f,\quad u|_{\partial \Omega}=0
	$$
	has a unique solution in $L^{2}(\bigwedge^k\Omega, \mu)$.
%	\end{comment}

	\section{$L^\infty$-dimension and proofs of Theorems \ref{thm:1.1} and \ref{thm:1.2}} \label{S:L}
	\setcounter{equation}{0}
		Let $M$ be a compact Riemannian $n$-manifold, and $\mu$ be a positive finite regular Borel measure on $M$ with compact support.  %Let $\Omega$ be an open subset of $\mathbb{R}^{d}$. Note that if the unit ball
	%$$B_{0}:=\left\{u \in \Gamma^\infty_c(\Omega):\|u\|_{W^{1,2}_0(\bigwedge^k\Omega)} \leq 1\right\}$$
	%is relatively compact in $L^2\big(\bigwedge^k\Omega,\mu\big)$, then condition (C1) holds and the embedding $\operatorname{dom}(\mathcal{E}) \hookrightarrow$ $L^2\big(\bigwedge^k\Omega,\mu\big)$ is compact. The following theorem, based on a result in [27], is crucial in establishing the relative compactness of $B_{0}$ in $L^2\big(\bigwedge^k\Omega,\mu\big)$
	 For each $p\in M$, the {\em injectivity radius of $M$ at $p$}, denoted as
	$${\rm inj}(p):=\inf\big\{r>0: \exp_p:B^{T_pM}(0, r)\rightarrow B^M(p, r)\,\, \text{is a diffeomorphism}\big\}.$$
	 Let
\begin{align*}
	{\rm inj}(M):=\inf\big\{ {\rm inj}(p),\,\, p\in M\big\}.
\end{align*}
Let $\epsilon\in(0,{\rm inj}(p))$. Then the exponential map $\exp_p:B^{T_pM}(0,\epsilon)\rightarrow B^M(p,\epsilon)$ is a diffeomorphism.  Every orthonormal basis $(b_i)$ for $T_pM$ determines a basis isomorphism  
\begin{align}\label{eq:4.1a}
E_p:T_pM  \rightarrow \mathbb{R}^n.
\end{align}
  Moreover, $E_p$ is an isometry. Let $B(0,\epsilon):=E_p(B^{T_pM}(0,\epsilon))$ and $S: B(0,\epsilon) \rightarrow B(z,\epsilon)$ be a similitude. Define a normal coordinate map
\begin{align*}
	\varphi:=S \circ E_p \circ \exp _{p}^{-1}: B^M(p,\epsilon) \rightarrow B(z,\epsilon)\subseteq\mathbb{R}^{n}.
\end{align*}
Since $M$ is compact, there exists  a finite open cover $\{B^M(p_i,\epsilon_i)\}_{i=1}^N$ of $M$, where $p_i\in M$, $\epsilon_i:=\epsilon_{p_i}\in(0,{\rm inj}(M))$, and $B^M(p_i,\epsilon_i)$ is a geodesic ball. Hence the mapping $\exp_{p_i}:B^{T_{p_i}M}(0,\epsilon_i)\rightarrow B^M(p_i,\epsilon_i)$ is a diffeomorphism.  For each $i=1,\ldots,N$, let $E_{p_i}: T_{p_i}M\rightarrow \mathbb{R}^n$ be defined in \eqref{eq:4.1a}, and let $B(0,\epsilon_i):=E_{p_i}(B^{T_{p_i}M}(0,\epsilon_i))$. Next we let $S_i: B(0,\epsilon_i) \rightarrow B(z_i,\epsilon_i)$ be similitudes so that  $B(z_i,\epsilon_i)$ are disjoint. Therefore we have a family of normal coordinate maps
\begin{align*}
	\varphi_i:=S_i \circ E_{p_i }\circ \exp _{p_i}^{-1}: B^M(p_i,\epsilon_i) \rightarrow B(z_i,\epsilon_i), \,\,i=1,\ldots,N,
\end{align*}
where  $\varphi_i(p_i)=z_i$ and the sets $\varphi_i(B^M(p_i,\epsilon_i))$ are disjoint.

	Let $\mathcal{B}:=\big\{u \in \Gamma^\infty_c(\bigwedge^kM):\|u\|_{W^{1,2}_0(\bigwedge^kM)} \leq 1\big\}$. Let $U\subseteq M$ be an open set. Let $\mathcal{B}_U:=\{u \in \Gamma^\infty_c(\bigwedge^kU):\|u\|_{W^{1,2}_0(\bigwedge^kU)} \leq 1\}$ and $\widetilde{\mathcal{B}}_U:=\{u \in C^\infty_c(U):\|u\|_{W^{1,2}_0(U)} \leq 1\}$.   The following theorem generalizes the compact embedding theorem of Maz'ja (see \cite[Section 8.8]{Maz'ja_1985}).
	\begin{thm}\label{thm:3.11}
	Let $M$ be a compact Riemannian $n$-manifold, $\Omega\subseteq M$ be an open set, and $\mu$ be a positive finite Borel measure on $M$ with $\rm{supp}(\mu)\subseteq \overline{\Omega}$ and $\mu(\Omega)>0$.  For $q>2$, the unit ball $\mathcal{B}$
		is relatively compact in $L^{q}(\bigwedge^kM, \mu)$ if and only if
	\begin{align}
		&\lim _{\delta \rightarrow 0^{+}} \sup _{w\in M ; r \in(0, \delta)} r^{1-n / 2} \mu\left(B^M(w,r)\right)^{1 / q}=0 \quad \text {\rm for } n>2,  \label{eq:3.18} 
		\end{align}
	and
		\begin{align}
		&\lim _{\delta \rightarrow 0^{+}} \sup _{w \in M ; r \in(0, \delta)}|\ln r|^{1 / 2} \mu\left(B^M(w,r)\right)^{1 / q}=0 \quad \text {\rm for } n=2\label{eq:3.19}.
	\end{align}
	\end{thm}

	\begin{proof}\,Assume that (\ref{eq:3.18}) and (\ref{eq:3.19}) hold. We will prove that for $q>2$, the unit ball $\mathcal{B}$
		is relatively compact in $L^{q}(\bigwedge^kM, \mu)$. 
		
		\noindent{\em Step 1.} Let $(U,\varphi)$ be a regular coordinate chart on $M$. Let $(\omega_j)_{j=1}^{\infty}$ be a bounded sequence in $\mathcal{B}_U$. 
	By \eqref{eq:1.1}, for each $j$, we write
		\begin{align*}
			\omega_j=\sum_{I\in  \mathcal{I}_{k,n} } \omega_j^Idx^I,
		\end{align*}
	where the coefficients $\omega_j^I$ are smooth real-valued functions on $U$. Hence
	\begin{align*}
		\|\omega_j\|_{W_0^{1,2}(\bigwedge^kU)}^2&=\int_U|\omega_j|^2\,dV_g+\int_U|\nabla_U\omega_j|^2\,dV_g \nonumber\\
		&=\int_U\Big(\sum_{I\in \mathcal{I}_{k,n}}|\omega_j^I|^2\,dV_g+\int_U\sum_{I,l}\bigg|\frac{\partial\omega_j^I}{\partial x^l}\bigg|^2\Big)\,dV_g \nonumber\\
		&=\sum_{I\in \mathcal{I}_{k,n}}\int_U\Big(|\omega_j^I|^2+\sum_{l}\big|\frac{\partial\omega_j^I}{\partial x^l}\big|^2\Big)\,dV_g \nonumber\\
		&=\sum_{I\in \mathcal{I}_{k,n}}\|\omega_j^I\|_{W_0^{1,2}(U)}^2.
	\end{align*}
As $\|\omega_j\|_{W_0^{1,2}(\bigwedge^kU)}\leq 1$, we have $\sum_{I\in \mathcal{I}_{k,n}}\|\omega_j^I\|_{W_0^{1,2}(U)}^2\leq 1$. Hence for each $I$, $\|\omega_j^I\|_{W_0^{1,2}(U)}^2\leq 1$.  By \cite[Lemma 18.6]{Tu_2011}, $\omega_j^I$ is smooth, and by definition, $\omega_j^I$ has compact support. Thus $\omega_j^I\in C^\infty_c(U)$ for each $I$.  Then for each $I$, $(\omega_j^I)_{j=1}^{\infty}$ is a bounded sequence in  $\widetilde{\mathcal{B}}_U$.  By \eqref{eq:3.18} and \eqref{eq:3.19}, we have
\begin{align*}
	&\lim _{\delta \rightarrow 0^{+}} \sup _{w\in U ; r \in(0, \delta)} r^{1-n / 2} \mu\left(B^M(w,r)\right)^{1 / q}=0 \quad \text {\rm for } n>2,   
\end{align*}
and
\begin{align*}
	&\lim _{\delta \rightarrow 0^{+}} \sup _{w \in U ; r \in(0, \delta)}|\ln r|^{1 / 2} \mu\left(B^M(w,r)\right)^{1 / q}=0 \quad \text {\rm for } n=2.
\end{align*}
Hence $\widetilde{\mathcal{B}}_U$ is relatively compact in $L^{q}(U, \mu)$ (see  \cite{Ngai-Ouyang_2023}).  Thus there exists $\omega_I\in L^q(U,\mu)$ such that for each $I$,
\begin{align*}
\lim_{j\rightarrow\infty}\|\omega_j^I-\omega_I\|_{L^{q}(U, \mu)}=0.
	\end{align*}
	Let $\omega:=\sum_{I\in  \mathcal{I}_{k,n} } \omega_Idx^I$. Then $\omega\in L^{q}(\bigwedge^kU, \mu)$ and
	\begin{align*}
		\lim_{j\rightarrow\infty}\|\omega_j-\omega\|_{L^{q}(\bigwedge^k U, \mu)}=\lim_{j\rightarrow\infty}\sum_{I\in \mathcal{I}_{k,n}}\|\omega_j^I-\omega_I\|_{L^{q}(U, \mu)}=0.
	\end{align*}
	Hence $\mathcal{B}_U$ is relatively compact in $L^{q}(\bigwedge^kU, \mu)$.

\noindent{\em Step 2}. By the compactness of $M$, we can choose a $C^\infty$ atlas $\{(W_\alpha,\varphi_\alpha)\}_{\alpha=1}^N$ satisfying $M\subseteq \cup_{\alpha=1}^N W_\alpha$.  By relabeling the $W_{\alpha}$, we can
choose geodesic balls $F_{\alpha}:=B^M(p_{\alpha},\epsilon_{\alpha})\subseteq W_{\alpha}$ satisfying
	$M\subseteq\bigcup_{\alpha=1}^N F_\alpha$ and $\overline{F}_{\alpha}\subseteq W_{\alpha}$.
Then $\{(F_\alpha,\varphi_\alpha|_{F_{\alpha}})\}_{\alpha=1}^N$ is a regular atlas on $M$.
Let $f_\alpha\in L^1(\bigwedge^kM)$. For each $\widetilde{\epsilon}_\alpha>0$, where $\alpha=2,\ldots,N$, there exists $\delta_\alpha\in(0,\epsilon_\alpha)$ such that for $\delta\in(0,\delta_\alpha)$,
			\begin{align}
			&\int_{\cup_{i=1}^{\alpha-1}(F_{\alpha}\cap(F_i\backslash B^M(p_i,\epsilon_i-\delta)))}|f_\alpha|^2\,dV_g<\frac{\widetilde{\epsilon}_\alpha}{2}\qquad\text{and} \label{eq:int1}\\
				&\int_{\cup_{i=1}^{\alpha-1}(F_{\alpha}\cap(F_i\backslash B^M(p_i,\epsilon_i-\delta)))}|\nabla f_\alpha|^2\,dV_g<\frac{\widetilde{\epsilon}_\alpha}{2}\label{eq:int}.
			\end{align}
		Let $\underline{\delta}:=\min\{\delta_\alpha:\alpha=2,\ldots,N\}$, $U_1:=F_1$, and $U_{\alpha}:=F_{\alpha}\backslash\overline{\cup_{i=1}^{\alpha-1} (F_\alpha\cap B^M(p_i,\epsilon_i-\underline{\delta}))}$ for $\alpha=2,\ldots,N$.
		Then $U_{\alpha}\subseteq F_{\alpha}$ for $\alpha=1,\ldots,N$, and $\{U_{\alpha}\}_{\alpha=1}^{N}$ is a  finite open cover of $M$. Let $\{g_{\alpha}\}_{\alpha=1}^{N}$ be a $C^\infty$ partition of unity subordinate to  $\{U_{\alpha}\}_{\alpha=1}^{N}$.
		Let $(\omega_j)_{j=1}^\infty$ be a bounded sequence in $\mathcal{B}$. Then we can write $\omega_j=\sum_\alpha g_{\alpha}\omega_j$, and thus ${\rm supp}(g_{\alpha}\omega_j)\subseteq U_{\alpha}$ and $g_{\alpha}\omega_j\in \Gamma^\infty_c\big(\bigwedge^kU_{\alpha}\big)$.
	Let $D_1:=U_1$ and $D_{\alpha}:=U_{\alpha}\backslash\overline{\cup_{i=1}^{\alpha-1}(F_\alpha\cap F_i)}$ for $\alpha=2,\ldots,N$. Then $D_\alpha\subseteq U_\alpha$ for $\alpha=1,\ldots,N$, and $\{D_\alpha\}_{\alpha=1}^N$ is a disjoint collection.
	As $\|\omega_j\|_{W^{1,2}_0(\bigwedge^kM)}\leq 1$, we have $\|g_\alpha \omega_j\|_{W^{1,2}_0(\bigwedge^kD_\alpha)}\leq 1$. Combining \eqref{eq:int1} and \eqref{eq:int}, we obtain $\|g_\alpha \omega_j\|_{W^{1,2}_0(\bigwedge^k(U_\alpha\backslash D_\alpha))}<\widetilde{\epsilon}_\alpha$. Hence $\|g_\alpha \omega_j\|_{W^{1,2}_0(\bigwedge^kU_\alpha)}\leq 1+\widetilde{\epsilon}_\alpha\leq2$.
		Thus $(g_{\alpha}\omega_j)$ is a bounded sequence in $ \underline{\mathcal{B}}_{{\alpha}}:=\{u \in \Gamma^\infty_c\big(\bigwedge^kU_{\alpha}\big):\|u\|_{W^{1,2}_0\left(\bigwedge^kU_{\alpha}\right)} \leq 2\}$. We know from Step 1 that for each $\alpha$, $\underline{\mathcal{B}}_{{\alpha}}$ is relatively compact in $L^{q}(\bigwedge^kU_\alpha, \mu)$.
		Hence there exists a subsequence  $(g_{1}\omega_{j_{m}}^{(1)}) \subset (g_1\omega_{j})$ converging to $\omega^{(1)}$ in $L^q(\bigwedge^kU_1, \mu)$. Similarly, there exists a subsequence  $(g_{2}\omega_{j_{m}}^{(2)}) \subset (g_{2}\omega_{j_{m}}^{(1)}) $ converging to $\omega^{(2)}$ in $L^q(\bigwedge^kU_2, \mu)$.  Continuing in this way, we see that the ``diagonal sequence" $(g_{\alpha}\omega_{j_{m}}^{(\alpha)})$ is a subsequence of $(g_{\alpha}\omega_j)$ such that for every $\alpha=1,\ldots,N$,
		\begin{align}\label{eq:cong}
		\lim_{m\rightarrow\infty}\|g_{\alpha}\omega_{j_{m}}^{(\alpha)}- \omega^{(\alpha)}\|_{L^q(\bigwedge^kU_{\alpha},\mu)}=0.
		\end{align}
		Write $\omega_{j_m}^{(N)}:=\sum_{\alpha=1}^{N}g_{\alpha}\omega_{j_{m}}^{(\alpha)}$ and $\omega:=\sum_{\alpha=1}^{N} \omega^{(\alpha)} \in L^q(\bigwedge^kM,\mu)$.
		Then
		\begin{align*}
			\|\omega_{j_m}^{(N)}-\omega\|^q_{L^q(\bigwedge^kM, \mu)}&=\int_{M}\Big|\sum_{\alpha=1}^{N} g_{\alpha}\omega_{j_{m}}^{(\alpha)}-\sum_{\alpha=1}^{N} \omega^{(\alpha)}\Big|^q\,d\mu\\
			&\leq C\sum_{\alpha=1}^{N} \int_{U_{\alpha}}|g_{\alpha}\omega_{j_{m}}^{(\alpha)}-\omega^{(\alpha)}|^q\,d\mu\\
			&\rightarrow 0 \qquad(\text{by}\,\,\eqref{eq:cong})
		\end{align*}
	as $m\rightarrow \infty$.	Hence  $\mathcal{B}$ is relatively compact in $L^q(\bigwedge^kM, \mu)$.

		Conversely, assume that for $q>2,$ the unit ball $\mathcal{B}$ is relatively compact in $L^q(\bigwedge^kM, \mu)$. We will show that \eqref{eq:3.18} and \eqref{eq:3.19} hold.		Using the compactness of $M$, we may select a finite
		 system $\{(V_i,\varphi_i)\}_{i=1}^N$ satisfying $M\subseteq \cup_{i=1}^NV_i$.  By relabeling the $V_i$, we can choose a regular atlas $\{(U_i,\varphi_i)\}_{i=1}^N$ satisfying $\overline{U}_i\subseteq V_i$ and $M\subseteq \cup_{i=1}^N U_i$. Fix $i$. Let $(u_m)$ be a bounded sequence in $\mathcal{B}_{U_i}$. Since $\mathcal{B}$ is relatively compact in $L^q(\bigwedge^kM, \mu)$, there exists $u\in L^q(\bigwedge^kM, \mu)$ such that 
			\begin{align}\label{eq:cont}
			\lim_{m \rightarrow \infty}\|u_m-u\|_{L^q(\bigwedge^kM, \mu)}=0.
			\end{align}
		Since $L^q(\bigwedge^kU_i, \mu)$ is complete, we can write $u=v_1+v_2$, where ${\rm supp}(v_1)\subseteq U_i$ and ${\rm supp}(v_2)\subseteq \partial U_i$. Suppose 
		$\int_{\partial U_i}|v_2|^qd\mu\neq 0,$
		then, since $\int_{\partial U_i}|u_m|^qd\mu= 0$, we have
		$$\lim_{m \rightarrow \infty}\|u_m-v_2\|_{L^q(\bigwedge^kM, \mu)}\neq 0.$$
		This contradicts \eqref{eq:cont}. Hence	$$\int_{\partial U_i}|v_2|^qd\mu= 0.$$ Thus $u=v_1$ $\mu$-a.e. in $\overline{U}_i$. Hence $\mathcal{B}_{U_i}$ is relatively compact in $L^q(\bigwedge^kU_i, \mu)$. For each $I$, let $\omega_j^I$ be a bounded sequence in $\widetilde{\mathcal{B}}_{U_i}$, i.e., $\omega_j^I\in C^\infty_c(U_i)$, and $\|\omega_j^I\|_{W_0^{1,2}(U_i)}\leq 1$.  For each $j$, we let $\omega_j=\sum_{I\in \mathcal{I}_{k,n}} \omega_j^Idx^I$. Then  $\omega_j\in \Gamma^\infty_c(\bigwedge^kU_i)$ and
		\begin{align*}
			\|\omega_j\|_{W_0^{1,2}(\bigwedge^kU_i)}=\sum_{I\in \mathcal{I}_{k,n}}\|\omega_j^I\|_{W_0^{1,2}(U_i)}\leq \sum_{I\in \mathcal{I}_{k,n}}1\leq C, \quad\text{for each}\,\, j.
		\end{align*}
		Let $$\overline{\mathcal{B}}_{i}:=\Big\{u \in \Gamma^\infty_c\Big({\bigwedge}^kU_i\Big):\|u\|_{W^{1,2}_0(\bigwedge^kU_i)} \leq C\Big\}.$$ Then $\omega_j$ is a bounded sequence in $\overline{\mathcal{B}}_{i}$. Since $\mathcal{B}_{U_i}$ is relatively compact in $L^q(\bigwedge^kU_i , \mu)$, $\overline{\mathcal{B}}_{i}$ is relatively compact in $L^q(\bigwedge^kU_i, \mu)$. Hence there exists $\omega\in L^q(\bigwedge^kU_i, \mu)$ such that 
			\begin{align*}
			\lim_{j\rightarrow\infty}\|\omega_j-\omega\|_{L^{q}(\bigwedge^k U_i, \mu)}=0.
		\end{align*}
		Write $\omega:=\sum_{I\in \mathcal{I}_{k,n}}\omega_Idx^I$. Then $\omega_I\in L^{q}(U_i, \mu)$ for each $I$. Hence for each $I$,
		\begin{align*}
			\lim_{j\rightarrow\infty}\|\omega_j^I-\omega_I\|_{L^{q}(U_i, \mu)}=0.
		\end{align*}
		It follows that $\widetilde{\mathcal{B}}_{U_i}$ is relatively compact in $L^q(U_i, \mu)$, and thus
		\begin{align}\label{1}
			&\lim _{\delta \rightarrow 0^{+}} \sup _{w\in U_i ; r \in(0, \delta)} r^{1-n / 2} \mu\left(B^M(w,r)\right)^{1 / q}=0 \quad \text {\rm for } n>2
		\end{align}
		and
		\begin{align}\label{2}
			&\lim _{\delta \rightarrow 0^{+}} \sup _{w \in U_i ; r \in(0, \delta)}|\ln r|^{1 / 2} \mu\left(B^M(w,r)\right)^{1 / q}=0 \quad \text {\rm for } n=2.
		\end{align}
Combining \eqref{1}, \eqref{2}, and the fact that $\{U_i\}_{i=1}^N$ is a finite open cover of $M$,  {\color{magenta}we see that} \eqref{eq:3.18} and \eqref{eq:3.19} hold.
	\end{proof}

	We say that $\mu$ is \textit{upper $s$-regular} for $s>0$ if there exists some $c>0$ such that for all $x \in {\rm supp}(\mu)$ and all $r$ satisfying $0 \leq r \leq \operatorname{diam}({\rm supp}(\mu))$,
	\begin{align*}
		\mu\left(B^{M}(x,r)\right) \leq c r^{s}.
	\end{align*}
	\textit{Lower $s$-regularity} is defined by reversing the inequality. The following theorem is well known; see, e.g., \cite[Lemma 3.1]{Ngai-Ouyang_2023}.
	\begin{lem}\label{lem:3.1}
		Let $\mu$ be a compactly supported finite positive Borel measure on a compact manifold $M$.
		\begin{enumerate}
			\item[(a)]  If $\mu$ is upper (resp. lower) $s$-regular for some $s>0$, then $\underline{\operatorname{dim}}_{\infty}(\mu) \geq s$ (resp. $\left.\overline{\operatorname{dim}}_{\infty}(\mu) \leq s\right)$.
			\item[(b)]  Conversely, if $\underline{\operatorname{dim}}_{\infty}(\mu) \geq s$ for some $s>0$, then $\mu$ is upper $\alpha$-regular for any $0<\alpha<s$.
		\end{enumerate}
	\end{lem}
	The proof of Theorem \ref{thm:3.12} mainly uses Theorem \ref{thm:3.11} and Lemma \ref{lem:3.1}(b); it is similar to that of \cite[Theorem 3.2]{Hu-Lau-Ngai_2006} and so is omitted.
	\begin{thm}\label{thm:3.12}Let $M$ be a compact Riemannian $n$-manifold. Let $n \geq 2$ and $2<q<\infty$, and let $\mu$ be a finite positive Borel measure on $M$ with compact support.  
		\begin{enumerate}
			\item[(a)] If $\underline{\operatorname{dim}}_{\infty}(\mu)>q(n-2) / 2$, then $\mathcal B$ is relatively compact in $L^{q}(\bigwedge^kM, \mu)$.
			\item[(b)] If $\underline{\operatorname{dim}}_{\infty}(\mu)<q(n-2) / 2$, then $\mathcal B$ is not relatively compact in $L^{q}(\bigwedge^kM, \mu)$.
		\end{enumerate}
	\end{thm}

We let $\Gamma^C(\bigwedge^k\overline{\Omega})$ be the subspace of $\Gamma^C(\bigwedge^k\Omega)$ consisting of all  $k$-forms in $\Gamma^C(\bigwedge^k\Omega)$ that are bounded and uniformly continuous on $\Omega$. Define
	$$\|u\|_{\Gamma^C(\bigwedge^k\overline{\Omega})}:=\sup_{x\in \overline{\Omega}}|u(x)|\qquad\text{and}\qquad\|u\|_{\infty}:=\ess \sup_{x\in \Omega}|u(x)|.$$
To prove Theorem \ref{thm:1.1}, we need the following theorem. 
	\begin{thm}\label{lem:3.13}
		Let $M$ be a compact Riemannian $1$-manifold, and let $\Omega \subseteq M$ be an open set. Then $W^{1,2}_0(\bigwedge^k\Omega)$ is compactly embedded in $\Gamma^C(\bigwedge^k\overline{\Omega})$.
	\end{thm}
	\begin{proof}
	Let $(U,\varphi)$ be a regular coordinate chart on $M$.  Let $(\omega_j)_{j=1}^{\infty}$ be a bounded sequence in $W^{1,2}_0(\bigwedge^kU)$.  By \eqref{eq:1.1}, for each $j$, we write
	\begin{align*}
		\omega_j=\sum_{I\in \mathcal{I}_{k,n}} \omega_j^Idx^I,
	\end{align*}
	where the coefficients $\omega_j^I$ are measurable functions on $U$. Hence
	\begin{align*}
		\|\omega_j\|_{W_0^{1,2}(\bigwedge^kU)}^2=\sum_{I\in \mathcal{I}_{k,n}}\|\omega_j^I\|_{W_0^{1,2}(U)}^2
	\end{align*}
and for each $I$, $\omega_j^I\in W_0^{1,2}(U)$.
\begin{comment}{\color{brown}	\begin{align*}
		\|\omega_j\|_{W_0^{1,2}(\bigwedge^kU)}^2&=\int_U|\omega_j|^2\,dV_g+\int_U|\nabla_U\omega_j|^2\,dV_g \nonumber\\
		&=\int_U\sum_{I\in \mathcal{I}_{k,n}}|\omega_j^I|^2\,dV_g+\int_U\sum_{I,l}\bigg|\frac{\partial\omega_j^I}{\partial x^l}\bigg|^2\,dV_g \nonumber\\
		&=\sum_{I\in \mathcal{I}_{k,n}}\int_U\Big(|\omega_j^I|^2+\sum_{l}\bigg|\frac{\partial\omega_j^I}{\partial x^l}\bigg|^2\Big)\,dV_g \nonumber\\
		&=\sum_{I\in \mathcal{I}_{k,n}}\|\omega_j^I\|_{W_0^{1,2}(U)}^2.
	\end{align*}}
\end{comment}
 Since $W^{1,2}_0(U)$ is compactly embedded in $C(\overline{U})$ (see \cite{Ngai-Ouyang_2023}), there exists $\omega_I\in C(\overline{U})$ such that for each $I$,
	\begin{align*}
		\lim_{j\rightarrow\infty}\|\omega_j^I-\omega_I\|_{L^{q}(U, \mu)}=0.
	\end{align*}
	Let $\omega:=\sum_{I\in \mathcal{I}_{k,n}}\omega^Idx^I$. Then $\omega\in \Gamma^C(\bigwedge^k\overline{U})$. Hence
	\begin{align*}
		\lim_{j\rightarrow\infty}\|\omega_j-\omega\|_{\Gamma^C(\bigwedge^k\overline{U})}\leq C\lim_{j\rightarrow\infty}\sum_{I\in \mathcal{I}_{k,n}}\|\omega_j^I-\omega^I\|_{C(\overline{U})}=0.
	\end{align*}
Thus $W^{1,2}_0(\bigwedge^kU)$ is compactly embedded in $\Gamma^C(\bigwedge^k\overline{U})$.

By the compactness of $\Omega$, we may select a finite
 system $\{(V_\alpha,\varphi_\alpha)\}_{\alpha=1}^N$ satisfying $\Omega\subseteq \cup_{\alpha=1}^NV_\alpha$.   By relabeling the $V_\alpha$, we can
	 choose open sets $U_{\alpha}\subseteq V_{\alpha}$ and a partition of unity $\{g_{\alpha}\}_{\alpha=1}^{N}$ satisfying: 
	\begin{align*}
		\Omega\subseteq\bigcup_{\alpha=1}^N U_\alpha,\,\, \overline{U}_{\alpha}\subseteq V_{\alpha},\,\,\supp(g_\alpha)\subseteq V_{\alpha},\,\,\text{and}\,\,\sum_{\alpha=1}^N g_{\alpha}(x)=1,\,\,x\in \Omega.
	\end{align*}
	Then $\{(U_\alpha,\varphi_\alpha)\}_{\alpha=1}^N$ is a regular atlas on $\Omega$. Let $(\omega_j)$ be a bounded sequence in $W^{1,2}_0(\bigwedge^k\Omega)$. Then we can write $\omega_j=\sum_\alpha g_{\alpha}\omega_j$, and thus ${\rm supp}(g_{\alpha}\omega_j)\subseteq U_{\alpha}$ and $g_{\alpha}\omega_j\in W^{1,2}_0(\bigwedge^kU_{\alpha})$.
	Thus $(g_{\alpha}\omega_j)$ is a bounded sequence in $ W^{1,2}_0(\bigwedge^kU_{\alpha})$. We know that for each $\alpha$, $  W^{1,2}_0(\bigwedge^kU_{\alpha})$ is relatively compact in $\Gamma^C(\bigwedge^k\overline{U}_{\alpha})$.
	Hence there exists a subsequence  $(g_{1}\omega_{j_{m}}^{(1)}) \subset (g_1\omega_{j})$ converging to $\omega^{(1)}$ in $\Gamma^C(\bigwedge^k\overline{U}_{1})$. Similarly, there exists a subsequence  $(g_{2}\omega_{j_{m}}^{(2)}) \subset (g_{2}\omega_{j_{m}}^{(1)}) $ converging to $\omega^{(2)}$ in $\Gamma^C(\bigwedge^k\overline{U}_{2})$.  Continuing in this way, we see that the ``diagonal sequence" $(g_{\alpha}\omega_{j_{m}}^{(\alpha)})$ is a subsequence of $(g_{\alpha}\omega_j)$ such that for every $\alpha=1,\ldots,N$,
	\begin{align}\label{eq:cong1}
	 \lim_{m\rightarrow\infty}\|g_{\alpha}\omega_{j_{m}}^{(\alpha)}- \omega^{(\alpha)}\|_{\Gamma^C(\bigwedge^k\overline{U}_{\alpha})}=0.
	 \end{align}
	Write $\omega_{j_m}^{(N)}:=\sum_{\alpha=1}^{N}g_{\alpha}\omega_{j_{m}}^{(\alpha)}$ and $\omega:=\sum_{\alpha=1}^{N} \omega^{(\alpha)} \in \Gamma^C(\bigwedge^k\overline{\Omega})$.
	Then
	\begin{align*}
		\lim_{m\rightarrow \infty}\|\omega_{j_m}^{(N)}-\omega\|_{\Gamma^C(\bigwedge^k\overline{\Omega})}&=	\lim_{m\rightarrow \infty}\Big\|\sum_{\alpha=1}^Ng_{\alpha}\omega_{j_m}^{(\alpha)}-\sum_{\alpha=1}^N\omega^{(\alpha)}\Big\|_{\Gamma^C(\bigwedge^k\overline{\Omega})}\\
		&\leq 	\lim_{m\rightarrow \infty}\sum_{\alpha=1}^{N} \|g_{\alpha}\omega_{j_m}^{(\alpha)}-\omega^{(\alpha)}\|_{\Gamma^C(\bigwedge^k\overline{\Omega})}
= 0\qquad(\text{by}\,\,\eqref{eq:cong1}).
	\end{align*}
	Hence $W^{1,2}_0(\bigwedge^k\Omega)$ is compactly embedded in $\Gamma^C(\bigwedge^k\overline{\Omega})$.
		\end{proof}

	\begin{proof}[Proof of Theorem \ref{thm:1.1}]\,
	 For the case $n=1$, we have by Theorem \ref{lem:3.13} that $W^{1,2}_0(\bigwedge^k\Omega)$ is compactly embedded in $\Gamma^C(\bigwedge^k\overline{\Omega})$. Let $(u_m)$ be a bounded sequence in $W^{1,2}_0(\bigwedge^k\Omega)$. There exists a subsequence $(u_{m_j})$ that converges to some $u\in \Gamma^C(\bigwedge^k\overline{\Omega})$. Hence
	\begin{align*}
		\lim\limits_{j\rightarrow\infty}\|u_{m_j}-u\|_{\Gamma^C(\bigwedge^k\overline{\Omega})}=\lim\limits_{j\rightarrow\infty}\|u_{m_j}-u\|_\infty=0.
	\end{align*}
	We know $u_{m_j}\in L^{2}(\bigwedge^k\Omega, \mu)$ and $u\in \Gamma^C(\bigwedge^k\overline{\Omega})\subseteq L^{2}(\bigwedge^k\Omega, \mu)$. Hence
	\begin{align*}
		\lim\limits_{j\rightarrow\infty} \int_\Omega|u_{m_j}-u|^2\,\,d\mu\leq\lim\limits_{j\rightarrow\infty}\int_\Omega\|u_{m_j}-u\|^2_\infty \,d\mu=\lim\limits_{j\rightarrow\infty}\|u_{m_j}-u\|^2_\infty\int_\Omega \,d\mu=0.
	\end{align*}
	Hence $(u_{m_j})$ converges to $u$ in $L^{2}(\bigwedge^k\Omega, \mu)$. It follows that $W^{1,2}_0(\bigwedge^k\Omega)$ is compactly embedded in $L^{2}(\bigwedge^k\Omega, \mu)$, and thus (PI) holds. Moreover, the embedding $\operatorname{dom}(\mathcal{E}) \hookrightarrow L^{2}(\bigwedge^k\Omega, \mu)$ is compact.  
	
	Using Theorem \ref{thm:3.12}, we can prove that the hypotheses of Theorem \ref{thm:1.1} are satisfied if $n\geq 2$. The proof is similar to that of \cite[Theorem 1.1]{Hu-Lau-Ngai_2006} and is omitted.
\end{proof}

		\begin{proof}[Proof of Theorem \ref{thm:1.2}]\, This is a direct consequence of Theorem \ref{thm:1.1} and \cite[Theorem B.1.13]{Kigami_2001}; we omit the proof.
		\end{proof}
	
	\section{Proof of Theorem \ref{thm:1.3}} \label{S:H}
\setcounter{equation}{0}
		
			In this section, we show that the classical Hodge theorem holds when $\Omega=M$ and the measure $\mu$ is absolutely continuous with a positive and bounded density.
			
			\begin{lem}\label{prop:5.1}
		Assume the same hypotheses of Theorem \ref{thm:1.3}. Then $\mathcal{N}_k=\{0\}$ and hence $\mathcal{N}_k^\perp=W_0^{1,2}\big({\bigwedge}^kM\big)$.
		\end{lem} 
	\begin{proof}
	By the definitions of $\iota$ and $\mathcal{N}_k$, for any $u\in \mathcal{N}_k$,
	\begin{align}\label{eq:jj1}
	\int_M \langle \iota(u),\iota(u)\rangle_g\,d\mu=0.
	\end{align}
	By assumption,	$d\mu=\rho\, dV_g,$ where $\rho>0$ is a density function. Thus $\langle \iota(u),\iota(u)\rangle_g=0$ $V_g$-a.e., and hence $\iota(u)=0$ $V_g$-a.e. As $\iota(u)=u$ $V_g$-a.e., we have
 $u=0$ $V_g$-a.e. This completes the proof.	
\begin{comment}{\color{brown}
\begin{align*}
&\int_M \langle \iota(u),\iota(u)\rangle_g\,\rho dV_g=0.\\
\Rightarrow \quad  &\langle \iota(u),\iota(u)\rangle_g\,\rho=0 \quad\,\, V_g-a.e.\\
\Rightarrow \quad  &\langle \iota(u),\iota(u)\rangle_g=0 \quad\,\, V_g-a.e.\\
\Rightarrow \quad  &\langle u,u\rangle_g=0 \quad\,\, V_g-a.e.\\
\Rightarrow \quad & u=0 \quad \quad \quad\, V_g-a.e.
\end{align*}}
\end{comment}
	\end{proof}
\begin{prop}\label{prop:1}
		Assume the same hypotheses of Theorem \ref{thm:1.3}. Then $\mathcal{H}^k_{\mu}(M)=\mathcal{H}^k(M)=\widetilde{\mathcal{H}}^k(M)$.
\end{prop}
\begin{proof} By Lemma \ref{prop:5.1}, we have $\mathcal{N}_k=\{0\}$ and  $\mathcal{N}_k^{\perp}=W^{1,2}_0(\bigwedge^kM)$. By \eqref{eq(2.2)}, for $\omega\in \operatorname{dom}(\Delta_{\mu}^k)$ and for any $u\in \operatorname{dom}(\mathcal{E})$, we have
		\begin{align*}
		\langle -\Delta_{\mu}^k\omega,u\rangle_{\mu}=\int_M \langle d\omega,du\rangle_g\,dV_g+\int_M \langle d^*\omega,d^*u\rangle_g\,dV_g.
	\end{align*}	
It follows from \eqref{no1} that $\omega\in \operatorname{dom}(\Delta^k)$ and hence		
		\begin{align}\label{eq:5.2*}
			\langle -\Delta_{\mu}^k\omega,u\rangle_{\mu}=\int_M \langle d\omega,du\rangle_g\,dV_g+\int_M \langle d^*\omega,d^*u\rangle_g\,dV_g=\langle -\Delta^k\omega,u\rangle_{L^2(\bigwedge^kM)}.
			\end{align}
		\begin{comment}	{\color{brown}\begin{align*}
			\langle \Delta_{\mu}^k\omega,u\rangle_{\mu}&=\int_M \langle d\omega,du\rangle_g\,dV_g+\int_M \langle d^*\omega,d^*u\rangle_g\,dV_g\nonumber\\
			&=\int_M \langle \omega,d^*du\rangle_g\,dV_g+\int_M \langle \omega,dd^*u\rangle_g\,dV_g\nonumber\\
			&=\int_M \langle \omega,d^*du+dd^*u\rangle_g\,dV_g\nonumber\\
			&=\int_M \langle \omega,\Delta^ku\rangle_g\,dV_g\nonumber\\
			&= \langle \omega,\Delta^ku\rangle\nonumber\\
			&=\langle \Delta^k\omega,u\rangle
		\end{align*}}
	\end{comment}
			\noindent{\em Step 1.} We first prove $\mathcal{H}^k_{\mu}(M)\subseteq \mathcal{H}^k(M)$.
		Let $\omega\in \mathcal{H}^k_{\mu}(M)$, i.e., $\omega\in \operatorname{dom}(\Delta_{\mu}^k)$ and  $\Delta_{\mu}^k\omega=0$. Then $\omega\in \operatorname{dom}(\Delta^k)$. For any  $u\in \operatorname{dom}(\mathcal{E})$, by \eqref{eq:5.2*} and the fact that $\Delta_{\mu}^k\omega=0$,   we have
		\begin{align*}
			\langle \Delta^k\omega,u\rangle_{L^2(\bigwedge^kM)}=0.
		\end{align*}
In particular,
\begin{align*}
	\langle \Delta^k\omega,\omega\rangle_{L^2(\bigwedge^kM)}= 0.
\end{align*}
Let $\omega=\sum_{m=1}^\infty c_m\psi_m$, where $c_m\in\R$ and $\left\{\psi_{m}\right\}_{m=1}^{\infty}$ is an orthonormal basis of $L^2(\bigwedge^kM)$. Then for all $u\in  \operatorname{dom}(\mathcal{E})$, 
	\begin{align*}
	\langle -\Delta^k\omega,u\rangle_{L^2(\bigwedge^kM)}=\mathcal{E}(\omega,u)=\sum_{m=1}^\infty\mathcal{E}(c_m\psi_m,u)= \Big\langle-\sum_{m=1}^\infty \lambda_mc_m\psi_m ,u\Big\rangle_{L^2(\bigwedge^kM)}.
	\end{align*}
	 Hence $\Delta^k\omega=\sum_{m=1}^\infty \lambda_mc_m\psi_m$ and therefore, as $\lambda_m\geq 0$ for all $m$, we have
\begin{align*}
	0&=	\langle \Delta^k\omega,\omega\rangle_{L^2(\bigwedge^kM)}=	\sum_{m=1}^\infty\lambda_mc_m^2\langle \psi_m,\psi_m\rangle_{L^2(\bigwedge^kM)}=\sum_{m=1}^\infty \lambda_mc_m^2.
\end{align*}
Assume that the 0-eigenvalue space is $d$-dimensional. For $m\geq d+1$, we have $\sum_{m=d+1}^\infty  \lambda_mc_m^2\\=0$ and  $\lambda_m>0$ for all $m\geq d+1$. Hence $c_m=0$ for all $m\geq d+1$. For $m=1,\ldots,d$, we have $\omega=\sum_{m=1}^{d}c_m\psi_m$. Hence $\Delta^k\omega=0$, and thus $\omega\in \mathcal{H}^k(M)$. This proves that $\mathcal{H}^k_{\mu}(M)\subseteq\mathcal{H}^k(M)$.

\noindent{\em Step 2.} Next, we show $\mathcal{H}^k(M)\subseteq\mathcal{H}^k_{\mu}(M)$. Let $\omega\in \mathcal{H}^k(M)$, i.e., $\omega\in \operatorname{dom}(\Delta^k)$ and $\Delta^k\omega=0$. Then $\omega\in \operatorname{dom}(\Delta_{\mu}^k)$. By \eqref{eq:5.2*}, for any $u\in \operatorname{dom}(\mathcal{E})$, $\langle \Delta_{\mu}^k\omega,u\rangle_{\mu}=0$.
In particular,
\begin{align*}
	\langle \Delta_{\mu}^k\omega,\omega\rangle_{\mu}=0.
\end{align*}
Applying a similar argument as that in  Step 1 shows that $\omega\in \mathcal{H}_\mu^k(M)$. 

\noindent{\em Step 3.} We now show $\mathcal{H}^k(M)\subseteq \widetilde{\mathcal{H}}^k(M)$. Let $\omega\in \mathcal{H}^k(M)$. Since $\operatorname{dom}(\Delta^k)\subseteq \operatorname{dom}(\mathcal{E})$, we have $\omega\in W^{1,2}_0(\bigwedge^kM)$.  By \eqref{eq:5.2*}, we have
 $\langle d\omega,du\rangle_{L^2(\bigwedge^kM)}=0$ and $\langle d^*\omega,d^*u\rangle_{L^2(\bigwedge^kM)}=0$. In particular, $\langle d\omega,d\omega\rangle_{L^2(\bigwedge^kM)}=0$ and $\langle d^*\omega,d^*\omega\rangle_{L^2(\bigwedge^kM)}=0$. Hence $\langle d\omega,d\omega\rangle_g=0$ and $\langle d^*\omega,d^*\omega\rangle_g=0$ $V_g$-a.e. Therefore $d\omega=d^*\omega=0$  $V_g$-a.e., proving that $\omega\in \widetilde{\mathcal{H}}^k(M)$.

\noindent{\em Step 4.} Finally we show $\widetilde{\mathcal{H}}^k(M)\subseteq \mathcal{H}^k(M)$. Let $\omega\in \widetilde{\mathcal{H}}^k(M)$.  i.e., $\omega\in W^{1,2}_0(\bigwedge^kM)$ and $d\omega=d^*\omega=0$. Then  for any $u\in \operatorname{dom}(\mathcal{E})$,
\begin{align*}
	\int_M \langle d\omega,du\rangle_g\,dV_g+\int_M \langle d^*\omega,d^*u\rangle_g\,dV_g=0.
\end{align*}
Hence, taking $0\in L^2(\bigwedge^kM,\mu)$, we have
\begin{align*}
	\int_M \langle d\omega,du\rangle_g\,dV_g+\int_M \langle d^*\omega,d^*u\rangle_g\,dV_g=\langle 0,u\rangle_{L^2(\bigwedge^kM)},\quad \forall\, u\in \operatorname{dom}(\mathcal{E}).
\end{align*}
By \eqref{no1}, we have $\omega\in \operatorname{dom}(\Delta^k)$.  By \eqref{eq:5.2*}, we have $\langle \Delta^k\omega,u\rangle_{L^2(\bigwedge^kM)}=0$.  In particular, we have $\langle \Delta^k\omega,\omega\rangle_{L^2(\bigwedge^kM)}=0$.
Using a similar argument as that in  Step 1 shows that $\Delta^k\omega=0$, and thus $\omega\in \mathcal{H}^k(M)$.

Combining the four steps above,  we have $\mathcal{H}^k_{\mu}(M)=\mathcal{H}^k(M)=\widetilde{\mathcal{H}}^k(M).$ 
\end{proof}

	\begin{proof}[Proof of the Theorem \ref{thm:1.3}] Combining Proposition \ref{prop:1} and the Hodge theorem (see \cite[Theorem 7.4.4]{Morrey_1966}) completes the proof.
		\end{proof}

The following example shows that  Theorem \ref{thm:1.3} need not hold if $\mu$ is not absolutely continuous.

\begin{exam}\label{exam:1}
	Let $\mathbb{T}^2=\mathbb{S}^1\times\mathbb{S}^1$ be the 2-torus, where $\mathbb{S}^1$ is the 1-sphere. Let $\mu$ be the Dirac measure on $\mathbb{T}^2$. Then $\mathcal{H}^1(\mathbb{T}^2)\neq\mathcal{H}^1_{\mu}(\mathbb{T}^2)$.
\end{exam}
\begin{proof}Let $\underline{\mathcal{H}}^1(\mathbb{T}^2):=\{\omega\in\Gamma^\infty(\bigwedge^1\mathbb{T}^2):\Delta^k\omega=0\}$. Since $\Gamma^\infty(\bigwedge^1\mathbb{T}^2)\subseteq W^{1,2}_0(\bigwedge^1\mathbb{T}^2)$, we have
	\begin{align}\label{eq:h2}
		\underline{\mathcal{H}}^1(\mathbb{T}^2)\subseteq \mathcal{H}^1(\mathbb{T}^2).
	\end{align}
	By the classical Hodge theorem (see, e.g., \cite[Theorem 47]{Petersen_2016}), we have $\underline{\mathcal{H}}^1(\mathbb{T}^2)\cong H^1_{\rm{dR}}(\mathbb{T}^2;\R)$, where $\cong$ denotes vector space isomorphism and $H^1_{\rm{dR}}(\mathbb{T}^2;\R)$ is the first de Rham cohomology of $\mathbb{T}^2$ . Since $H^1_{\rm{dR}}(\mathbb{T}^2;\R)\cong \R^2$ (see, e.g., \cite[Section 28]{Tu_2011}), we have $\dim (\underline{\mathcal{H}}^1(\mathbb{T}^2))=2$.
Combining this with \eqref{eq:h2}, we have
\begin{align}\label{eq:h1}
	\dim (\mathcal{H}^1(\mathbb{T}^2))\geq2.
\end{align}
	Note that if $\omega=\eta$ $\mu$-a.e., then $\|\omega-\eta\|_\mu=0$. It follows that  $\dim (L^2(\bigwedge^1\mathbb{T}^2,\mu))=1$. Since $\mathcal{H}^1_{\mu}(\mathbb{T}^2)$ is a subspace of $L^2(\bigwedge^1\mathbb{T}^2,\mu)$, 	we have, by \eqref{eq:h1}, $\mathcal{H}^1(\mathbb{T}^2)\neq\mathcal{H}^1_{\mu}(\mathbb{T}^2)$.
	%	\begin{align}\label{eq:h2}
	%	\dim(\mathcal{H}^1_\mu(\mathbb{T}^2))=1.
%	\end{align}
%By the definition of $\mathcal{H}^1_{\mu}(\mathbb{T}^2)$, we have  $\mathcal{H}^1_{\mu}(\mathbb{T}^2)\subseteq L^2(\bigwedge^1\mathbb{T}^2,\mu)$.
\end{proof}

	\end{document}